\documentclass[12pt]{amsart}

\input epsf.tex

\usepackage{amsmath}
\usepackage{amsthm}
\usepackage{amsfonts}
\usepackage{amssymb}
\usepackage{graphicx}
\usepackage{latexsym}

\usepackage{amsmath,amsthm,amsfonts,amssymb}

\usepackage{epsfig}

\theoremstyle{plain}
\newtheorem{thm}{Theorem}[section]
\newtheorem{prop}[thm]{Proposition}
\newtheorem{lem}[thm]{Lemma}
\newtheorem{cor}[thm]{Corollary}
\newtheorem{conj}{Conjecture}

\theoremstyle{definition}
\newtheorem{defn}{Definition}
\theoremstyle{remark}
\newtheorem{remark}{Remark}

\advance \topmargin by -\headheight
\advance \topmargin by -\headsep
\evensidemargin \oddsidemargin
\def\sC{{\mathcal C}}
\def\C{{\mathbb C}}

\def\F{{\mathbb F}}
\def\cF{{\mathcal F}}

\def\G{{\cal G}}
\def\sG{{\mathcal G}}
\def\G{{\Gamma}}

\def\H{{\mathbb H}}
\def\Haus{\textrm{Haus}}
\def\HD{\mathcal{HD}}

\def\int{{\textrm{int \,}}}

\def\Isom{\textrm{Isom}}

\def\tmu{{\tilde \mu}}

\def\PSL{{\textrm{PSL}}}

\def\R{{\mathbb R}}
\def\vr{{\vec r}}

\def\supp{\textrm{support}}

\def\sT{{\mathcal T}}

\def\chix{{\raise.5ex\hbox{$\chi$}}}

\def\vv{{\vec v}}

\def\vw{{\vec w}}

\def\tY{{\tilde Y}}

\def\Z{{\mathbb Z}}

\begin{document}

\title{Free Groups in Lattices}
\author{Lewis Bowen}
\begin{abstract}
Let $G$ be any locally compact unimodular metrizable group. The main result of this
paper, roughly stated, is that if $F<G$ is any finitely generated free group and $\Gamma
< G$ any lattice, then up to a small perturbation and passing to a finite index subgroup,
$F$ is a subgroup of $\Gamma$. If $G/\Gamma$ is noncompact then we require additional
hypotheses that include $G=SO(n,1)$.

\end{abstract}
\maketitle
\noindent {\bf MSC}: 20E07, 20F65, 20F67, 22D40, 20E05.\\
{\bf Keywords}: free group, surface group, Kleinian group, limit set.
\noindent


\section{Introduction}

Consider the following general problem. Let $\Gamma$ be a discrete cocompact subgroup of a locally compact metrizable unimodular group $G$. Can we use information about the subgroups of $G$ to infer the existence of subgroups of $\Gamma$ satisfying prescribed properties? 

For example, suppose $G=PSL_2(\C)$, the group of orientation-preserving isometries of hyperbolic $3$-space. $G$ contains a large variety of surface subgroups; that is, subgroups that are each isomorphic to the fundamental group of a closed surface of genus at least 2. It is a well-known open problem whether $\Gamma$, an arbitrary discrete cocompact subgroup of $G$, contains a surface subgroup. 

This paper investigates the following strategy: given a subgroup $F<G$, attempt to change $F$ in some small way so that the resulting subgroup $F'$ lies in $\Gamma$ and retains important properties of $F$. For example, we would like $F'$ to be isomorphic to a finite-index subgroup of $F$ and the embedding $F'<G$ to have asymptotic geometric properties close to those of $F$.

To be precise, let $F$ be an abstract group and $\phi:F \to G$ a homomorphism. Let $S\subset F$ be a finite symmetric generating set. Let $d$ be a left-invariant proper metric on $G$ inducing its topology. For $\epsilon>0$, we say that a map $\phi_\epsilon:F \to G$ is an {\it $\epsilon$-perturbation} of $F$ if 
$$d( \phi_\epsilon(fs), \phi_\epsilon(f)\phi(s))  \le \epsilon$$
for all $f \in F$ and $s \in S$. $\phi_\epsilon$ need not be a homomorphism. Indeed, we do not even require that it maps the identity element to the identity element.

For example, if $G=\R$, $F=\Z$ and $\phi:\Z \to \R$ is the homomorphism $\phi(n)=n\tau$ for some number $\tau >0$ then $\phi_\epsilon:\Z \to \R$ need only satisfy $|\phi_\epsilon(n+1) - \phi_\epsilon(n) - \tau| \le \epsilon$ for all $n\in \Z$. 

We say that $\phi_\epsilon: F\to G$ is {\it virtually a homomorphism} if there exists a finite
index subgroup $F'<F$ such that
$$\phi_\epsilon(f_1f_2)=\phi_\epsilon(f_1)\phi_\epsilon(f_2) \hspace{0.4in} \forall f_1 \in F', \, f_2 \in F.$$
If, in addition, $\phi_\epsilon(F') <\Gamma$ then we say it is {\it virtually a homomorphism into $\Gamma$}.

\begin{thm}[Main theorem: Uniform case]\label{thm:uniform}
Let $G, \Gamma, d, F, S$ and $\phi$ be as above. Suppose $F$ is free and $S$ is a symmetric free generating set for $F$. Then for every $\epsilon>0$ there exists an $\epsilon$-perturbation $\phi_\epsilon$ of $\phi$ that is virtually a homomorphism into $\Gamma$.
\end{thm}
I do not know if the theorem remains true if $F$ is required to be a surface group instead.

\begin{thm}[Main theorem: Nonuniform Case]\label{thm:nonuniform}
Let $G=SO(n,1)$, $\Gamma< G$ be a discrete group with finite covolume, $F < G$ a convex cocompact free group, $S$ a symmetric free generating set for $F$ and $\epsilon>0$. Then there exists an $\epsilon$-perturbation $\phi_\epsilon$ of the inclusion map $\phi:F\to G$ that is virtually a homomorphism into $\Gamma$.
\end{thm}

\subsection{Asymptotic Geometry}\label{subsection:about}

Next we make a precise claim to the effect that, under special negative-curvature conditions, asymptotic geometric properties do not change much under an $\epsilon$-perturbation. The terms used below are standard. For the reader's convenience, they are listed in subsection \ref{sub:defns}.

Let $(X,d)$ be a proper Gromov-hyperbolic space. Let $\partial X$ denote the Gromov boundary of $X$. Given a subset $Y \subset X \cup \partial X$, let $L(Y)=\partial X \cap \overline{Y}$ where $\overline{Y}$ denotes the closure of $Y$ in $X \cup \partial X$. If $\phi_\epsilon:H \to \Isom(X)$ is a map (where $H$ is an abstract group), then define the limit set of $\phi_\epsilon$ by $L(\phi_\epsilon)=L(\phi_\epsilon(H)p)$ where $p\in X$ is any point. $L(\phi_\epsilon)$ does not depend on the choice of $p$.

Fix a visual metric $d_\partial$ on $X$ with respect to some point $p\in X$. Let $d_{\Haus}$ denote the Hausdorff distance on closed subsets of $\partial X$ with respect to $d_\partial$. Let $\HD$ denote the Hausdorff dimension of subsets of $\partial X$ with respect to $d_\partial$.



\begin{thm}\label{thm:asymptotic}
Let $H$ be an abstract group with finite symmetric generating set $S$ and $\phi:H \to \Isom(X)$ be an injective homomorphism onto a quasi-convex cocompact subgroup of the group of isometries of $X$. Let $d_{\Isom(X)}$ be a left-invariant metric on $\Isom(X)$ inducing the topology of uniform convergence on compact sets. Then for all $C>0$, there exists an $\epsilon_0 >0$ such that if $0 \le \epsilon \le \epsilon_0$ and $\phi_\epsilon:H \to G$ is an $\epsilon$-perturbation of $\phi$ then
\begin{enumerate}
\item $\phi_\epsilon$ is injective,
\item $d_{\Haus}\Big(L(\phi_\epsilon), L(\phi)\Big) \le C$ and
\item if $\phi_\epsilon$ is virtually a homomorphism then $\Big| \HD(L(\phi_\epsilon) - \HD(L(\phi)) \Big| \le C$.
\end{enumerate}

 \end{thm}


\subsection{Applications}
Let $\H^n$ denote $n$-dimensional hyperbolic space. $\Isom^+(\H^n)$, the group of orientation-preserving isometries of $\H^n$ is identified with $SO(n,1)$. If $H$ is any subgroup of $\Isom^+(\H^n)$, let $D_{free}(H)$ denote the set
of all numbers $d$ such that $d=\HD(L(F))$, the Hausdorff dimension of the limit set of a free convex cocompact subgroup $F<H$.

\begin{thm}\label{thm:dense}
If $\Gamma$ is a lattice in $\Isom^+(\H^n)$ then $\overline{D_{free}(\Gamma)} = \overline{D_{free}(\Isom^+(\H^n))}$.
\end{thm}
This follows immediately from the above theorems. 

\begin{remark}\label{remark:dimension}
It is easy and well-known that $D_{free}(\Isom^+(\H^2))=(0,1)$. From work of Thurston and others on
geometrically infinite free groups it can be proven that $D_{free}(\Isom^+(\H^3))=(0,2)$. It is not
known whether these results extend to $\Isom^+(\H^n)$ for $n \ge 4$.
\end{remark}

A {\it surface group} is a group isomorphic to the fundamental group of a closed surface
of genus at least 2.

\begin{remark}
Let $D_s(\Gamma)$ denote the set of Hausdorff dimensions of limit sets of surface subgroups of $\Gamma$. In general, $D_s(\Gamma)$ is very mysterious. For example, in case $\Gamma<SO(3,1)$ is a uniform lattice, it is not known whether $\overline{D_s(\Gamma)}=[1,2]$ for any $\Gamma$ or for all $\Gamma$ (even for well-studied lattices like the fundamental group of the figure-eight knot complement). It is even unknown whether $\overline{D_s(\Gamma)}$ contains an interval for any such $\Gamma$ or for all $\Gamma$.
\end{remark}

Let $P$ be a right-angled, compact Coxeter polyhedron in
$\mathbb{H}^3$,
and let $\G(P) \subset \Isom(\mathbb{H}^3)$ be the group
generated by reflections in the faces
of $P$.

 \begin{thm}[Ma07]
 Every finitely generated free subgroup of $\G(P)$ is contained in a subgroup $Q<\G(P)$ isomorphic to the fundamental group of a closed orbifold of negative Euler characteristic.
\end{thm}

The corollary below was pointed out by Joseph Masters.
\begin{cor}
For every $d<2$ there exists a surface subgroup $H_d < \G(P)$ such that the Hausdorff
dimension of the limit set of $H_d$ is at least $d$.
\end{cor}
\begin{proof}
By remark \ref{remark:dimension} and theorem \ref{thm:dense} there exists a finitely generated free subgroup $F<\Gamma$ with $\HD(L(F)) \ge d$. By the theorem above, there exists a group $Q$ with $F<Q<\Gamma$ such that $Q$ is isomorphic to the fundamental group of a closed orbifold of negative Euler characteristic. Because $F<Q$, $\HD(L(Q))\ge d$. It is well-known that $Q$ contains a finite-index surface subgroup $Q'$. Because $Q'$ has finite index in $Q$, $\HD(L(Q'))=\HD(L(Q))\ge d$.
\end{proof}

The next application regards mapping class groups. Let $S$ denote an oriented closed
hyperbolic surface and $Mod(S)=\pi_0(Homeo^+(S))$ its group of orientation preserving
self-homeomorphisms up to isotopy. Let $D_{free}(S)$ denote the set of all numbers of the
form $\HD(L(F))$, where $F$ is a free convex compact subgroup of $Mod(S)$ and $\HD(L(F))$ is the
Hausdorff dimension of the limit set of $F$ on the boundary at infinity of Teichm\"uller space. The result below
was first observed by Chris Leininger.

\begin{thm}
$[0,1] \subset \overline{D_{free}(S)}$.
\end{thm}

\begin{remark}
It seems highly unlikely that $[0,1]=\overline{D_{free}(S)}$ but I do not know that this is
false.
\end{remark}

\begin{proof}
There exist disks (called Teichm\"uller disks) contained in Teichm\"uller space that are isometric with the hyperbolic plane $\H^2$. By [Ve89], there exists such a disk whose stabilizer in the mapping class group is a lattice in $\PSL_2(\R)$, the group of all (orientation-preserving) isometries of $\H^2$. The circle at infinity of the hyperbolic plane piecewise projectively embeds in the Thurston boundary of Teichm\"uller space. So Hausdorff dimension in the circle is the same as in Thurston's boundary with respect to the natural piecewise projective structure. This theorem now follows from theorem \ref{thm:dense}.
\end{proof}

Theorem \ref{thm:dense} plays a key role in the two recent papers, [LLR08] and [La08]. The first relates LERF with the Lubotzky-Sarnak conjecture and the second proves that Kleinian groups that contain noncyclic finite subgroups are either virtually free or contain a surface subgroup.

\subsection{Organization}

To prove theorem \ref{thm:uniform}, the action of $F$ on $\Gamma \backslash G$ is embedded into a symbolic dynamical system over $F$. A result in the symbolic dynamics over a free group implies that this larger system contains a periodic point. That point is used to construct the $\epsilon$-perturbation $\phi_\epsilon$. Thus, the core of the proof is a symbolic dynamics result. That result is stated and proven in \S \ref{sec:sd}. In \S \ref{sec:uniform}, theorem \ref{thm:uniform} is proven. 

A more general symbolic dynamics result is proven in section \ref{sec:sd2}. That result is used to prove theorem \ref{thm:nonuniform} in \S \ref{sec:nonuniform}. Section \ref{section:about} contains the proof of theorem \ref{thm:asymptotic}. In the last section, we explore attempts to replace free groups with surface groups in theorem 1. We prove that a continuous version of the symbolic dynamics result of \S \ref{sec:sd} does not hold. But we conjecture that with an additional hypothesis, the result of \S \ref{sec:sd} is true for surface groups. We show that this conjecture implies the surface subgroup conjecture.

{\bf Acknowledgements}.
I'd like to thank Joe Masters and Chris Leininger for the applications above.
Conversations with Chris Leininger were helpful in formulating the proof of theorem
\ref{thm:nonuniform}. I'm grateful to Marc Lackenby for pointing out several errors in a previous version and making helpful suggestions. I'd like to thank Alan Reid and Darren Long for useful conversations that have improved the paper. Last but not least, the referee has been most helpful by carefully going over the paper and making recommendations.

\section{Symbolic Dynamics}\label{sec:sd}
The core of the proof of theorem \ref{thm:uniform} is a result in the symbolic dynamics over a finitely generated free group $\F$. To describe it, let $S$ be a symmetric free generating set for $\F$. Let $\sG=(V,E)$ be a multi-graph. Assume that each edge is directed and has a label in $S$. We will associate to $\sG$ a dynamical system over $\F$.

Let $V^{\F}$ be the set of all functions $x:\F \to V$ with the topology of uniform convergence on finite sets. Let $X=X_\sG$ be the set of all $x\in X$ such that for all $f\in \F$ and $s\in S$, there is a directed edge in $\sG$ from $x(f)$ to $x(fs)$ labeled $s$. $X$ is called the {\it graph subshift determined by $\sG$}. $\sG$ is the {\it constraint graph} of $X$. 

To buttress the analogy with the classical symbolic dynamics over the integers, an element of $V^{\F}$ is called a {\it treequence}. For $g \in \F$, the {\it shift operator} $\sigma_g: V^{\F} \to V^{\F}$ is defined by $(\sigma_g x)(f)=x(g^{-1}f)$. This defines an action of $\F$ on $V^{\F}$. $X_\sG$ is closed and shift-invariant. Thus $\F$ acts on $X_\sG$.

The symmetry group of a treequence $x\in X$ is defined by $F_x=\{f\in \F~|~\sigma_fx=x\}$. If it has finite index in $\F$ then $x$ is said to be {\it periodic}. A measure $\mu$ on $X$ is {\it shift-invariant} if $\mu(\sigma_g E) = \mu(E)$ for all
$g\in {\F}$ and all measurable sets $E$. Equivalently, $\mu$ is {\it ${\F}$-invariant}.



\begin{thm}\label{thm:core}
Let $X \subset V^{\F}$ be the graph subshift defined by a {\em finite} graph $\sG$. If there exists a shift-invariant Borel probability measure $\mu$ on $X$, then there exists a periodic treequence $x \in X$. Moreover, if for some $v\in V$, $\mu\big(\{x \in X|x(id)=v\}\big)>0$ then there exists a periodic treequence $x\in X$ with $x(id)=v$. Here, $id$ denotes the identity element.
\end{thm}

\begin{remark}This theorem is the only place in the proof of theorem
\ref{thm:uniform} where
the fact that ${\F}$ is a free group is used. Theorem \ref{thm:core} is not true if ${\F}$
is replaced by the fundamental group of a closed surface of genus at least 2. However, it
is possible that under additional hypotheses on $X$, it remains true. See \S
\ref{section:surface} for further discussion.
\end{remark}

\begin{remark} If ${\F}$ has rank at least 2 then, because ${\F}$ is nonamenable, it is possible that there are no shift-invariant Borel probability measures on $X$. 
\end{remark}

The proof of theorem \ref{thm:core} given next is essentially the same as the proof given in \cite{Bo03}, where it was introduced. The idea comes from a simple observation: if $x:\F \to V$ is periodic, then it descends to a function $\bar{x}:F_x \backslash \F \to V$ by $\bar{x}(F_x f)=x(f)$. To prove theorem \ref{thm:core}, we will construct a periodic treequence by reversing this procedure. That is, we will construct a right action of $\F$ on a finite set $K$ and a function ${\bar x}:K \to V$ such that some ``lift'' of this function (defined by $x(f) = {\bar x}(k_1\cdot f)$ where $k_1 \in K$ is fixed), is a periodic treequence in $X$. This is accomplished through a study of weights on the constraint graph $\sG$ which are defined next.


Let $V=V(\sG), E(\sG)$ denote the vertex set and edge set of $\sG$ respectively. We write $(v,w;s)$ to denote the edge in $E(\sG)$ from $v$ to $w$ labeled $s$ (where $v,w \in V$, $s\in S$). A {\it weight} on $\sG$ is a function $W:V(\sG) \cup E(\sG) \to [0,\infty)$ such that 
\begin{eqnarray}\label{eqn:weight1}
\forall ~v\in V(\sG), s\in S, && W(v) = \sum_{w \in V} W(v,w;s) = \sum_{w \in V} W(w,v;s);\\
\forall ~v,w \in V(\sG), s\in S, && W(v,w;s)=W(w,v;s^{-1}).\label{eqn:weight2}
\end{eqnarray}
The first sum above is over all $w \in V$ such that $(v,w;s) \in E(\sG)$ and the second sum is over all $w\in V$ such that $(w,v;s) \in E(\sG)$. A weight is nontrivial if it is not identically zero.

\begin{lem}\label{lem:Wmu}
Let $\mu$ be a shift-invariant Borel probability measure on $X$.  Let
\begin{eqnarray*}
W_\mu(v) &=& \mu\big( \{ x\in X~|~ x(id) = v\} \big),\\
W_\mu(v,w;s) &=& \mu\big( \{ x\in X~|~ x(id) = v, ~x(s)=w\} \big).
\end{eqnarray*}
Then $W_\mu$ is a nontrivial weight on $\sG$.
\end{lem}

\begin{proof}
The equation
$$W_\mu(v) = \sum_{w \in V} W_\mu(v,w;s) = \sum_{w \in V} W_\mu(w,v;s)$$
holds by additivity of $\mu$. The equation $W_\mu(v,w;s)=W_\mu(w,v;s^{-1})$ is true because $\mu$ is shift-invariant and
$\sigma_s\big( \{ x\in X~|~ x(id) = v, ~x(s)=w\} \big)=  \{ x\in X~|~ x(id) = w, ~x(s^{-1})=v\}.$ 
\end{proof}

\begin{lem}\label{lem:core}
Let $W:V(\sG) \cup E(\sG) \to [0,\infty)$ be a nontrivial weight. If, for some $v_1 \in V(\sG)$, $W(v_1)>0$, then there exists a periodic treequence $x\in X$ with $x(id)=v_1$.
\end{lem}

\begin{proof}
The weight equations (\ref{eqn:weight1}, \ref{eqn:weight2}) are linear equations. So the space of all weights on $\sG$ is the intersection of a certain linear subspace of $\R^{V(\sG) \cup E(\sG)}$ with the positive orthant. Because these equations have coefficients in $\Z$, the existence of the weight $W$ with $W(v_1)>0$ implies the existence of a rational weight $W'$ such that $W'(v_1)>0$. Rational means that $W'(v)$ and $W'(e)$ are rational numbers for every $v\in V(\sG)$ and $e\in E(\sG)$. In fact, we may assume that $W'(v)$ and $W'(e)$ are integers for every $v\in V(\sG)$ and $e \in E(\sG)$ since multiplying a weight by a positive scalar does not change the fact that it is a weight.

Since $S$ is a symmetric free generating set for $\F$, we may write $S=\{s_1,\ldots, s_r, s_1^{-1},\ldots, s_r^{-1}\}$. Let $S^+=\{s_1,\ldots,s_r\}$. So $\F=\langle s_1,\ldots, s_r\rangle$. 

By the above, we may assume that $W(v)$ and $W(e)$ are integers for every $v \in V(\sG)$ and $e\in E(\sG)$. For each $v\in V(\sG)$, let $K(v)$ be a set with $|K(v)| = W(v)$. For every $v\in V(\sG)$ and $s\in S^+$, choose a partition $\{K_+(v,w;s)\}_{w\in V}$ of $K(v)$ so that $|K_+(v,w;s)| = W(v,w;s)$. This is possible because of the weight equations (\ref{eqn:weight1}, \ref{eqn:weight2}). Similarly, choose a partition $\{K_-(v,w;s)\}_{w\in V}$ of $K(v)$ so that $|K_-(v,w;s)| = W(w,v;s^{-1})$. For each $v,w \in V$ and $s\in S^+$, choose a bijection $b_{v,w;s}: K_+(v,w;s) \to K_-(w,v;s)$. 

Let $K$ be the disjoint union $\bigsqcup_{v\in V} K(v)$. For $s\in S^+$, the bijections $\{b_{v,w;s}\}_{v,w \in V}$ form a permutation of $K$ as follows. For $k \in K(v)$, let $k\cdot s := b_{v,w;s}(k)$ where $w \in V$ is the unique element such that $k \in K_+(v,w;s)$. Since $S$ freely generates $\F$ as a group, this defines a right action of $\F$ on $K$. 



Let ${\bar x}:K\to V(\sG)$ be the function ${\bar x}(k)=v$ if $k\in K(v)$. Now we can choose a periodic treequence $x \in X$ as follows. Let $k_1 \in K_{v_1}$. For $f \in \F$, define $x(f) = \bar{x}(k_1\cdot f)$. Because $K$ is finite, the stabilizer $F_{k_1}:=\{ f\in \F~|~ k_1\cdot f=k_1\}$ has finite index in $\F$. Since $\sigma_f x=x$ for every $f\in F_{k_1}$, $x$ is periodic. Also $x(id)={\bar x}(k_1)=v_1$. Apriori, $x$ is only in $V^\F$. Let us check that $x \in X$. Let $f\in F, s\in S^+$. Let $l=k_1\cdot f$, $v=x(f)=\bar{x}(l)$ and $w=x(f\cdot s)=\bar{x}(l\cdot s)$. Then $b_{v,w;s}(l)=l\cdot s$. Thus $W(v,w;s)>0$ which implies $(v,w;s) \in E(\sG)$. So $x\in X$.
\end{proof}
Theorem \ref{thm:core} is an immediate consequence of the lemmas above.


\section{The Uniform Case}\label{sec:uniform}

In this section, theorem \ref{thm:uniform} is proven. So all the hypotheses of theorem \ref{thm:uniform} are assumed. Briefly, the proof goes as follows. ${\F}$
acts on $\Gamma \backslash G$ on the left by $f(\Gamma g ):=\Gamma g \phi(f^{-1})$. The space
$\Gamma \backslash G$ is partitioned into subsets of
small diameter. This partitioning is used to embed $\Gamma \backslash G$ into a graph subshift $X$. Haar measure on $\Gamma \backslash G$ pushes forward to a
shift-invariant measure on $X$. Theorem \ref{thm:core}
implies the existence of a periodic treequence $x \in X$. This treequence is ``decoded'' to produce the required $\epsilon$-perturbation. 


\subsection{The Graph Subshift}\label{subsection:graph}
We consider $\Gamma \backslash G$ with the quotient metric $\bar{d}$ defined by $\bar{d}(\Gamma g_1,\Gamma g_2) = \min_{\gamma \in \Gamma} d(\gamma g_1, g_2)$. Recall that $S$ is a symmetric free generating set for $\F$. Let $\delta>0$ be such that for all $g_1, g_2 \in G$ with $d(g_1, id)<\delta$ and $d(g_2, id)<\delta$, if $s\in S$ then
$$d\big(g_1\phi(s)g_2, \phi(s)\big) < \epsilon.$$
Let $V=\{v_1,v_2,\ldots, v_n\}$ be a Borel partition
of $\Gamma \backslash G$ into sets $v_i$ of diameter less than
$\delta$. Assume that each $v_i$ has positive Haar measure. Let $\sG$ be the graph with vertex set $V=\{v_1,\ldots, v_n\}$ and edges defined as follows. For each $v, w \in V$, if there exists elements $p \in v$, $q\in w$ and $s\in S$ such that $p \phi(s) = q$ then there is a directed edge in $\sG$ from $v$ to $w$ labeled $s$. There are no other edges.

Let $X \subset V^{\F}$ be the graph subshift determined by $\sG$.

\subsection{Perturbations from treequences}\label{perturbations}

We will choose, for each $x\in X$, an $\epsilon$-perturbation $\phi_x$ of $\phi$. To get started, choose a basepoint $p_i \in v_i$ for each $i$. Without loss of generality, assume $p_1 = \Gamma$.

 If there is an edge $e=(v,w)$ in $\sG$ labeled $s$ then there exists points $p \in v, q\in w$ such that $p \phi(s)=q$. Let $p_v, q_w$ be the basepoints of $v$ and $w$ respectively. Because $v$ and $w$ each have diameter at most $\delta$, there exists elements $g_v, g_w \in G$ such that $d(g_v, id)<\delta$, $d(g_w,id)<\delta$, $p_vg_v=p$ and $q_w g_w=q$. Let $\psi_e=g_v \phi(s) g_w^{-1}$. Note that $p_v \psi_e = p_v g_v \phi(s) g_w^{-1}=q_w$. By choice of $\delta$, $d(\psi_e, \phi(s)) < \epsilon$.

There is an edge $e'= (w,v)$ in $\sG$ labeled $s^{-1}$. Choose $\psi_{e'}$ so that $\psi_{e'} = \psi_e^{-1}$. 

Let $x \in X$. For $f \in \F$, represent $f$ as $f=t_1\cdots t_m$ for some $t_i \in S $. Let $t_0=id$. Let $\phi_x(f)=\psi_{e_1} \cdots \psi_{e_m}$ where $e_i$ is the edge from $x(t_0 \cdots t_{i-1})$ to $x(t_1 \cdots t_{i})$ labeled $t_i$. This is independent of the choice of representation of $f$ because of the choice for $\psi_{e'}$ above and because $\F$ is freely generated by $S$.

To show that $\phi_x$ is an $\epsilon$-perturbation of $\phi$ we will need the next lemma.

\begin{lem}\label{lem:start}
A map $\eta:\F \to G$ is an $\epsilon$-perturbation of $\phi$ if for any sequence $t_1, \ldots, t_m \in S$, there exist elements $t'_i \in G$ such that $d(\phi(t_i),t'_i)<\epsilon $ and $\eta(t_1\cdots t_i)=t'_1\cdots t'_i$ for all $i=1\ldots m$.
\end{lem}

\begin{proof}
Let $f \in \F$ and $s \in S$. Then there exists a sequence $t_1,\ldots, t_m \in S $ such that $f=t_1\cdots t_m$. Let $t_{m+1}=s$. By hypothesis, for $1\le i \le m+1$, there exist elements $t'_i \in G$ such that $d(\phi(t_i),t'_i)<\epsilon$ and $\eta(t_1\cdots t_i)= t'_1\cdots t'_i$ for all $i=1\ldots m+1$. Thus
$$d\big(\eta(fs), \eta(f)\phi(s)\big) =  d\big(t'_1\cdots t'_{m+1},t'_1\cdots t'_m \phi(s)\big)=d(t'_{m+1},\phi(s))=d(t'_{m+1},\phi(t_{m+1}))<\epsilon.$$

\end{proof}

\begin{cor}
For every $x \in X$, the map $\phi_x:\F \to G$ is an $\epsilon$-perturbation of $\phi$.
\end{cor}

\subsection{Embedding $\Gamma \backslash G$ in $X$}\label{embedding}

Let $L: \Gamma\backslash G \to V$ be the labeling map. That is, $L(\Gamma g)= v$ if $\Gamma g \in v$. For each $\Gamma g   \in \Gamma \backslash G$, let $\Psi(\Gamma g) \in V^{\F}$ be the treequence $\Psi(\Gamma g)(f)=L( \Gamma g\phi(f))$.

${\F}$  acts on the left on $\Gamma \backslash G$ by: $f \cdot \Gamma g:=\Gamma g \phi(f^{-1})$. This
action preserves $\mu$, the normalized Haar measure on $\Gamma \backslash G$. $\Psi$ is
equivariant with respect to the left action of ${\F}$. The image of $\Psi$ lies inside $X$, so $\Psi_*(\mu)$ is a shift-invariant Borel probability measure on $X$. By theorem \ref{thm:core}, there exists a periodic treequence $z \in X$. Indeed, since $\mu(v_1)>0$, $\Psi_*(\mu)\big(\{ x \in X ~|~ x(id) = v_1\}\big) >0$. Thus, there exists a periodic treequence $z \in X$ such that $z(id)=v_1$. 

To finish the proof of theorem \ref{thm:uniform}, we claim that $\phi_z$ is virtually a homomorphism into $\Gamma$. We will need the following lemma.

\begin{lem}
For any $x \in X$ and $f,g\in \F$, $\phi_x(f)\phi_{f^{-1}x}(g) = \phi_x(fg).$
\end{lem}

\begin{proof}
Let $f=t_1\cdots t_m$, $g=u_1\cdots u_k$ with $t_i, u_i \in S$. Let $t_0=u_0=id$. Let $e_i$ be the edge in $\sG$ from $x(t_0 \cdots t_{i-1})$ to $x(t_1 \cdots t_{i})$ labeled $t_i$. Let $e'_i$ be the edge in $\sG$ from $x(fu_0 \cdots u_{i-1})$ to $x(fu_1 \cdots u_{i})$ labeled $u_i$. By definition, $\phi_x(f)= \psi_{e_1} \cdots \psi_{e_m}$ and  $\phi_x(fg)= \psi_{e_1} \cdots \psi_{e_m}\psi_{e'_1} \cdots \psi_{e'_k}$.

Note that $e'_i$ is the edge from $(f^{-1}x)(u_0 \cdots u_{i-1})$ to $(f^{-1}x)(u_1 \cdots u_{i})$ labeled $u_i$. Therefore, $\phi_{f^{-1}x}(g)=\psi_{e'_1} \cdots \psi_{e'_k}=\phi_x(f)^{-1}\phi_x(fg)$.
\end{proof}

Now, let $F_z< \F$ be the symmetry group of $z$. Since $z$ is periodic, $F_z$ has finite index in $\F$. Let $f\in F_z$ and $g \in \F$. Then
$$\phi_z(fg) = \phi_z(f)\phi_{f^{-1}z}(g) = \phi_z(f)\phi_{z}(g).$$
This shows that $\phi_z$ is virtually a homomorphism. To show that $\phi_z(F_z) < \Gamma$, we need the next lemma.

\begin{lem}
Let $x\in X$ be such that $x(id)=v_1$. Then, for any $f \in \F$, $\Gamma \phi_x(f) $ is the basepoint of $x(f) \in V$.
\end{lem}

\begin{proof}
Let $f=t_1\cdots t_m$ with $t_i \in S$. Let $t_0=id$. Let $e_i$ be the edge in $\sG$ from $x(t_0 \cdots t_{i-1})$ to $x(t_1 \cdots t_{i})$ labeled $t_i$. By definition, $\phi_x(f)= \psi_{e_1} \cdots \psi_{e_m}$. 

Let $q_i$ be the basepoint of $x(t_0 \cdots t_{i})\in V$. By definition of $\psi_e$, $q_{i-1}\psi_{e_{i}} = q_{i}$. So $q_0 \phi_x(f) = q_m$. Since $x(id)=v_1$, $q_0 = \Gamma$. So $\Gamma \phi_x(f) = q_m$ is the basepoint of $x(t_0 \cdots t_{m})=x(f)$.
\end{proof}

The lemma implies that for  $f\in F_z$, $\Gamma \phi_z(f) $ is the basepoint of $z(f) \subset \Gamma \backslash G$. But $z(f)=z(id) = v_1$ since $f \in F_z$. So the basepoint of $z(f)$ is $\Gamma$. That is, $\Gamma \phi_z(f) =\Gamma$. Of course, this implies $\phi_z(F_z) < \Gamma$. This completes the proof of theorem \ref{thm:uniform}.

\section{Subshifts determined by infinite graphs}\label{sec:sd2}

The proof of theorem \ref{thm:nonuniform} follows the same ideas as the proof of theorem \ref{thm:uniform}. However, because $\Gamma \backslash G$ may be noncompact, it is necessary to work with infinite partitions and therefore, with subshifts determined by infinite graphs. Theorem \ref{thm:core} does not apply in this case. So we generalize theorem \ref{thm:core} to certain infinite-graph subshifts. This is used in the next section to prove theorem \ref{thm:nonuniform}. To begin, we need some definitions.

\begin{defn}\label{defn:connected}
Recall that $S \subset \F$ is a finite symmetric free generating set. Let $\cF$ be the Cayley graph of $\F$. It has vertex set $\F$ and for every $f\in\F$ and $s\in S$, there is a directed edge from $f $ to $fs$ labeled $s$. The {\it induced subgraph} of a set $F \subset \F$ is the largest subgraph of $\cF$ with vertex set $F$. If it connected then we say $F$ is {\it $S$-connected}. An {\it $S$-connected component} of a set $F \subset \F$ is an $S$-connected subset $D \subset \F$ that is maximal among all $S$-connected subsets of $F$ with respect to inclusion.
\end{defn}

\begin{thm}\label{thm:core2}
Let $X \subset V^{\F}$ be a graph subshift determined by a graph $\sG=(V,E)$. Suppose that there is a finite set $A \subset V$ and a shift-invariant Borel probability measure $\mu$ on $X$ such that for $\mu$-almost every $x\in X$, every $S$-connected component the of set $x^{-1}(V-A) \subset \F$ is finite. Then there exists a periodic treequence in $X$. If for some $a_1\in A$, $\mu\big(\{x\in X~|~x(id)=a_1\}\big)>0$, then there exists a periodic treequence $x\in X$ with $x(id)=a_1$.

\end{thm}

The rest of this section proves this theorem. The next section shows how to apply this result to obtain theorem \ref{thm:nonuniform}. To prove this theorem, we show that there exists a weight supported on a finite subgraph of $\sG$ and then invoke lemma \ref{lem:core}. To do this, we represent the weight $W_\mu$ as a sum of functions that correspond to $x^{-1}(A)$ and the connected components of $x^{-1}(V-A)$ for $x\in X$. Then a simple convex geometric argument yields the existence of the desired weight. We will need some definitions.

Let $W_\mu$ be as defined in lemma \ref{lem:Wmu}. Let $W'_\mu:V \cup E \to [0,\infty)$ be the function defined by ``truncating $W_\mu$'' off of $V-A$. To be precise:
\begin{itemize}
\item $W'_\mu(a)=W_\mu(a)$ for $a\in A$,
\item $W'_\mu(a, b;s) =W_\mu(a, b;s)$ for $a,b \in A$ and $s\in S$,
\item $W'_\mu(v)=0$ for $v \in V-A$,
\item $W'_\mu(v,w;s) = 0$ if either $v \in V-A$ or $w \in V-A$.
\end{itemize}
$W'_\mu$ is not a weight in general. We will write $W_\mu$ as a sum of $W'_\mu$ and some other functions, defined next.

The {\it outer boundary} of a set $C \subset \F$ is the set of all elements $f \in \F$ such that $f \notin C$ but $f$ is adjacent to an element in $C$ (i.e., $\exists s \in S$ such that $fs \in C$). It is denoted by $\partial_o C$.

Let $Z$ be the collection of all functions $z:D_z \to V$ such that
\begin{itemize}
\item $D_z \subset \F$ is finite,
\item if $C_z=z^{-1}(V-A)$ then $C_z$ is connected and $D_z = C_z \cup \partial_o C_z$.
\end{itemize}

For $z \in Z$, let $[z] \subset V^\F$ be the set of all functions $x:\F \to V$ such that there is an $f\in \F$ satisfying 
\begin{itemize}
\item $f^{-1} \in C_z$,
\item $x(fd)=z(d)$ for all $d\in D_z$.
\end{itemize}

Recall that $(v,w;s)$ denotes the edge in $E(\sG)$ from $v$ to $w$ labeled $s$ (where $v,w \in V$ and $s\in S$) if one exists. For $z\in Z$, define a function $W_z:V(\sG)\cup E(\sG) \to [0,\infty)$ as follows.
\begin{itemize}
\item For $a \in A$, let $W_z(a)=0$.
\item For $v \in V-A$, let 
$$W_z(v) =\mu\big( \{x \in X| x(id)=v, x\in [z] \} \big).$$

\item For $v\in V-A$, $w\in V$ and $s\in S$, let $$W_z(v,w;s) = \mu\big( \{x \in X| x(id)=v, x(s)=w, x\in [z] \} \big).$$

\item For $a \in A, v \in V-A, s\in S$ let $W_z(a,v;s) = W_z(v,a;s^{-1})$.
\item Let $W_z(a,b;s)=0$ for any $a,b \in A$ and $s\in S$.
\end{itemize}
The function $W_z$ is not a weight in general. Since the domain of each $z \in Z$ is finite, $W_z$ is supported on a finite subgraph of $\sG$ (i.e., the subset of $V\cup E$ on which $W_z$ is nonzero is finite). Choose a subcollection $Z' \subset Z$ such that for all $z\in Z$ there exists a unique $z'\in Z'$ with $[z]=[z']$. 
\begin{lem}
$$W_\mu = W'_\mu + \sum_{z \in Z'} W_z.$$
\end{lem}

\begin{proof}
This follows immediately from the definitions and the hypothesis on $\mu$. Note that the sets $[z]$ for $z\in Z'$ are pairwise disjoint.
\end{proof}

Let $\vv \in \R^{A \times S}$ be the vector
$$\vv(a,s) = \sum_{b\in V-A} W_\mu(a, b;s).$$

For $z\in Z$, let $\vv_z \in \R^{A \times S}$ be the vector 
$$\vv_z(a,s) = \sum_{b\in V-A} W_z(a,b;s).$$

The lemma above implies $\vv(a,s) = \sum_{z \in Z'} \vv_z(a,s)$. The next lemma enables us to replace this sum with a finite sum.

\begin{lem}
Let $R=\{\vr_i\}_{i=1}^\infty$ be a sequence of nonnegative vectors in $\R^k$ for some $k < \infty$. Let $\vr_\infty$ be the sum $\vr_\infty:=\sum_{i=1}^\infty \vr_i$. If $\vr_\infty \in \R^k$ (i.e., every component of $\vr_\infty$ is finite) then there exists an $N>0$ and nonnegative coefficients $t_1, \ldots, t_N$ such that $\vr_\infty = \sum_{i=1}^N t_i \vr_i.$
\end{lem}

\begin{proof}
 If $U$ is a set of vectors in $\R^k$, then the {\it positive cone} of $U$ is the set of all vectors that can be expressed as $\sum_{i=1}^\infty c_i u_i$ with $c_i \ge 0$ and $u_i \in U$. Let $C$ be the closure of the positive cone of $R=\{\vr_i\}_{i=1}^\infty$.

If the interior of $C$ is empty, then
$C$ lies inside some linear subspace of $\R^k$ of positive codimension in which $C$ has nonempty interior. After replacing $\R^k$ with this subspace if necessary, it may be assumed that the interior of $C$ is nonempty.

\underline{Claim 1}: $\vr_\infty$ is in the interior of $C$.

Proof: Suppose for a contradiction that $\vr_\infty$ is on the boundary of $C$. Because $C$ is convex, there exists a supporting hyperplane $\Pi$ to $C$ at $\vr_\infty$. So, $\vr_\infty \in \Pi$ and $C$ lies in one
of the closed halfspaces determined by $\Pi$. Since the interior of $C$ is nonempty, there exists
$\vr_j\in R$ such that $\vr_j \notin \Pi$. 

Since $\vr_j \notin \Pi$ and $\vr_\infty \in \Pi$, it follows that the vector $\vr_\infty -\vr_j$ lies in the open half-space determined by $\Pi$ that does not contain $\vr_j$, i.e., the half-space that does not contain the interior of $C$. But $\vr_\infty - \vr_j = \sum_{i\ne j} \, \vr_i$ is contained in $C$. This contradiction proves the claim.


Let $C_n$ be the positive cone of $\{\vr_1,\ldots,\vr_n\}$.

\underline{Claim 2}: If $\vw$ is any point in the interior of $C$, then there exists $N>0$ such that $\vw \in C_N$.

Proof: Suppose for a contradiction that $\vw$ is not in $C_n$ for any $n$. Because $C_n$ is
convex there exists a hyperplane $\Pi_n$ containing $\vw$ that has trivial intersection
with $C_n$. The sequence $\{\Pi_n\}$ has a subsequential limit hyperplane $\Pi$ (with respect to the Hausdorff topology). Because $\{C_n\}$ is
an increasing sequence of convex sets, it follows that $\Pi$ does not intersect any of
the $C_n$'s. Therefore $C$ must be contained in one of the closed half-spaces determined by
$\Pi$. But this contradicts the hypothesis that $\vw$ is in the interior of $C$.

The two claims imply the lemma.
\end{proof}

By the lemma, there exists a finite collection $Z'' \subset Z'$ and nonnegative coefficients $t_z$ (for $z \in Z''$) such that $\vv = \sum_{z\in Z''} t_z \vv_z.$ Define a weight $W$ on $\sG$ by $W = W'_\mu + \sum_{z\in Z''} t_z W_z.$

\begin{lem}
$W$ is weight on $\sG$. It is supported on a finite subgraph. If for some $a_1 \in A$, $W_\mu(a_1)>0$ then $W(a_1)>0$.
\end{lem}

\begin{proof}
It is immediate from the definitions that $W(v,w;s)=W(w,v;s^{-1})$ for all $v,w \in V$ and $s\in S$. We must show that equation (\ref{eqn:weight1}) holds for $W$. This is accomplished in two separate cases.

{\bf Case 1}. Let $a\in A$ and $s\in S$. We must show that $W(a) = \sum_{b \in V} W(a,b;s).$
First, 
$$\sum_{b \in V-A} W(a,b;s) =\sum_{b \in V-A} \sum_{z\in Z''} t_z W_z(a,b;s)= \sum_{z\in Z''} t_z\vv_z(a,s)= \vv(a,s) = \sum_{b\in V-A} W_\mu(a,b;s).$$
If $b\in A$ then $W(a,b;s)=W'_\mu(a,b;s)=W_\mu(a,b;s)$. So, 
\begin{eqnarray*}
W(a) &=& W_\mu(a) = \sum_{b \in A} W_\mu(a,b;s) + \sum_{b \in V-A} W_\mu(a,b;s)\\
&=& \sum_{b \in A} W(a,b;s) + \sum_{b \in V-A} W(a,b;s)=  \sum_{b \in V} W(a,b;s).
\end{eqnarray*}

{\bf Case 2}. Let $v\in V-A$ and $s\in S$. We must show that $W(v) = \sum_{w \in V} W(v,w;s).$ For any $z\in Z$,
\begin{eqnarray*}
W_z(v)&=&   \mu\big( \{x \in X| x(e)=v, x\in [z] \} \big)= \sum_{w \in V}  \mu\big( \{x \in X| x(e)=v, x(s)=w, x\in [z] \} \big)\\
&=& \sum_{w \in V} W_z(v,w;s).
\end{eqnarray*}
Thus,
$$W(v) = \sum_{z\in Z''} t_z W_z(v)
 = \sum_{z\in Z''} t_z \sum_{w \in V} W_z(v,w;s)
=  \sum_{w \in V} W(v,w;s).$$

From cases 1 and 2 and the fact that $W(v,w;s)=W(w,v;s^{-1})$, it follows that for any $v\in V$ and $s\in S$,
$$W(v)= \sum_{w \in V} W(v,w;s) =  \sum_{w \in V} W(w,v;s^{-1}).$$
Thus $W$ is a weight. It is supported on a finite subgraph because $Z''$ is finite and each $W_z$ for $z \in Z''$ has finite support. 
\end{proof}

Theorem \ref{thm:core2} now follows from the lemma above and lemma \ref{lem:core}.

\section{The Nonuniform Case}\label{sec:nonuniform}

The key ingredient to proving theorem \ref{thm:nonuniform} from theorem \ref{thm:core2} is the next lemma. Fix a nonuniform lattice $\Gamma<G=SO(n,1)=\Isom^+(\H^n)$. After passing to a finite index subgroup, we may assume, by Selberg's lemma, that $\Gamma$ is torsion-free. So $\H^n/\Gamma$ is a manifold. Let $\phi:\F \to G$ be an injective homomorphism onto a convex cocompact subgroup of $G$. Note that $SO(n,1)$ is unimodular, so Haar measure on $G$ induces a $G$-invariant probability measure on $\Gamma \backslash G$.



\begin{lem}
For any $\delta>0$ there exists a Borel partition $V=\{v_1,v_2,\ldots\}$ of $\Gamma \backslash G$ into sets of diameter at most $\delta$ and a finite set $A \subset V$ such that the following holds.

Let $L: \Gamma\backslash G \to V$ be the labeling map. So $L(\Gamma g)= v$ if $\Gamma g \in v$. For each $\Gamma g \in \Gamma \backslash G$, let $\Psi(\Gamma g) \in V^{\F}$ be the treequence $\Psi(\Gamma g)(f)=L\big( \Gamma g\phi(f)\big)$. Let $\mu$ be the probability measure on $\Gamma \backslash G$ induced by Haar measure on $G$. Then for $\mu$-almost every $\Gamma g \in \Gamma \backslash G$, every $S$-connected component of $\Psi(\Gamma g)^{-1}(V-A)$ is finite. (See the previous section for the definition of $S$-connected).
\end{lem}

\begin{proof}
 Let $C \subset \partial \H^n $ be the cusp set of $\Gamma$. That is, $C$ is the set of all points $c\in \partial \H^n $ such that there exists a nontrivial parabolic element $g \in \Gamma$ with $gc=c$. It is countable since $\Gamma$ is countable.

Identify $\H^n$ with $G/K$ where $K<G$ is a maximal compact subgroup. Then $\H^n/\Gamma$ is identified with $\Gamma \backslash G/K$. Let $Q: \H^n \to \H^n/\Gamma$ be the quotient map. By Margulis' thin/thick decomposition of $\H^n/\Gamma$, there exists a compact set $T \subset \H^n/\Gamma$ such that if $T^c$ denotes the complement of $T$ in $\H^n/\Gamma$ then every connected component of $Q^{-1}(T^c)$ is a horoball. Each of these horoballs has a unique limit point $l \in \partial \H^n $ that is contained in the cusp set $C$. 

Choose a set $T$ as above so that $\Gamma K \in T$. Also choose $T$ so that for any $g\in G$ and $s\in S$, $\Gamma g K$ and $\Gamma g \phi(s) K $ cannot be in different connected components of $T^c$. To accomplish this, let $d_\Gamma$ be the distance function on $\H^n/\Gamma$ given by 
$$d_\Gamma(\Gamma g_1 K, \Gamma g_2 K) = \min \big\{d(\gamma_1 g_1 k_1, \gamma_2 g_2 k_2)~ | ~\gamma_1,\gamma_2 \in \Gamma, k_1,k_2 \in K\big\}.$$ 
Choose $T$ so large so that if $p,q$ are in different components of $T$ then $d_\Gamma(p,q) > \max_{s\in S} d(\phi(s), id)$. Then for any $g\in G$, $d_\Gamma(\Gamma g K, \Gamma g \phi(s)K) < d(\phi(s),id)$. Thus $\Gamma g K$ and $\Gamma g\phi(s) K$ cannot be in different components of $T^c$.


Let $\pi:\Gamma \backslash G \to \H^n/\Gamma$ be the projection map. Choose a Borel partition $V=\{v_1,v_2, \ldots\}$ of $\Gamma \backslash G$ into sets of diameter at most $\delta$ so that for some $A \subset V$, $\pi^{-1}(T) = \bigcup_{a\in A} a$. 


{\bf Claim:} If $g\in G$ is such that some $S$-connected component of $\Psi(\Gamma g)^{-1}(V-A)$ is infinite then $gL(\phi(\F)) \cap C \ne \emptyset$. 

Proof. If some $S$-connected component of $\Psi(\Gamma g)^{-1}(V-A)$ is infinite then there exists a set $F_0 \subset \F$ that is $S$-connected, infinite and $\Psi(\Gamma g)(f) \in V-A$ for all $f\in F_0$. The last condition implies that $\Gamma g \phi(f) K \notin T$ for all $f\in F_0$. Because $F_0$ is $S$-connected, the choice of $T$ implies that there is a component $H_0$ of $Q^{-1}(T^c)$ such that $g\phi(F_0)K \subset H_0$. Since $\phi(F_0)$ is infinite and discrete, there exist a point $l \in L(\phi(\F)) \subset \partial \H^n $ in the closure of $\phi(F_0) K$. Then $gl$ is in the closure of $H_0$. Therefore, $gl$ is in the cusp set $C$. This proves the claim.

For $c \in C$, let $G_c =\{g \in G~|~ g^{-1}c \in L(\phi(\F))\}$. Because $\phi(\F)$ is a convex cocompact free group, $L(\phi(\F))$ has measure zero in $\partial \H^n $ (with respect to Lebesgue measure). Therefore, $G_c$ has Haar measure zero. Since $C$ is countable, $\bigcup_{c\in C} G_c$ has Haar measure zero. By the claim, the set of all $g\in G$ such that some $S$-connected component of $\Psi(\Gamma g)^{-1}(V-A)$ is infinite is contained in $\bigcup_{c \in C} G_c$. So it has measure zero. This proves the lemma.
\end{proof}

As in subsection \ref{subsection:graph}, let $\delta>0$ be such that for all $g_1, g_2 \in G$ with $d(g_1, id)<\delta$ and $d(g_2, id)<\delta$, if $s\in S$ then $d\big(g_1\phi(s)g_2, \phi(s)\big) < \epsilon.$

Let $\mu$ be the probability measure on $\Gamma \backslash G$ induced by Haar measure on $G$. Choose a Borel partition $V=\{v_1,v_2,\ldots\}$ of $\Gamma \backslash G$ into sets $v_i$ of diameter less than
$\delta$ such that $\Gamma \in v_1$, $\mu(v_1)>0$ and $V$ satisfies the conclusion of the lemma above. 

Let $\sG$ be the graph with vertex set $V=\{v_1,v_2,  \ldots\}$ and edges defined as follows. For each $v, w \in V$, if there exists elements $p \in v$, $q\in w$ and $s\in S$ such that $p \phi(s) = q$ then there is a directed edge in $\sG$ from $v$ to $w$ labeled $s$. There are no other edges.

Let $X \subset V^{\F}$ be the graph subshift determined by $\sG$. As in subsection \ref{perturbations}, for every $x \in X$, there is an $\epsilon$-perturbation $\phi_x:\F \to G$ of $\phi$. 

Define $\Psi:\Gamma \backslash G \to X$ as in the lemma above. As in subsection \ref{embedding}, this map commutes with the action of $\F$. Therefore $\Psi_*(\mu)$ is a shift-invariant Borel probability measure on $X$. By the lemma above and theorem \ref{thm:core2}, there exists a periodic treequence $z \in X$. Indeed, since $\mu(v_1)>0$, $\Psi_*(\mu)\big(\{ x \in X ~|~ x(id) = v_1\}\big) >0$. Thus, there exists a periodic treequence $z \in X$ such that $z(id)=v_1$. As in subsection \ref{embedding}, $\phi_z$ is a $\epsilon$-perturbation of $\phi$ that is virtually a homomorphism into $\Gamma$. This proves theorem \ref{thm:nonuniform}.

\section{Asyptotic Geometric Properties}\label{section:about}
The goal of this section is to prove theorem \ref{thm:asymptotic}. For the reader's convenience, the next subsection defines the terms used in the statement of the theorem.

\subsection{Definitions}\label{sub:defns}

A nice reference for all the concepts below is [BH99].

\begin{defn}
Let $(X,d)$ be a metric space. A {\it geodesic} is a map $\gamma:I \to X$ where $I \subset\R$ is an interval and $d(\gamma(t),\gamma(t')) = |t-t'|$ for all $t', t\in [a,b]$. It is a {\it geodesic ray} from $x\in X$ if, in addition, $I=[0,\infty)$ and $\gamma(0)=x$. A {\it geodesic from $x$ to $y \in X$} is a geodesic of the form $\gamma:[a,b]\to X$ with $\gamma(a)=x$ and $\gamma(b)=y$. The image of $\gamma$ is a {\it geodesic segment} from $x$ to $y$ and is commonly denoted $[x,y]$ (although it depends on $\gamma$ and not just on $x$ and $y$). A {\it geodesic triangle} with vertices $x,y,z\in X$ is a union of three geodesic segments; one from $x$ to $y$, one from $y$ to $z$ and one from $z$ to $x$. $(X,d)$ is a {\it geodesic space} if between any two points $x,y \in X$ there exists a geodesic between them. 

For $\delta \ge 0$, a metric space $(X,d)$ is {\it $\delta$-hyperbolic} if it is a geodesic space and for every geodesic triangle $\Delta \subset X$, each side of $\Delta$ is contained in the $\delta$-neighborhood of the union of the other two sides. If $(X,d)$ is $\delta$-hyperbolic for some $\delta \ge 0$, then we say it is {\it Gromov-hyperbolic}.
\end{defn}


\begin{defn}[Gromov boundary]
Let $(X,d)$ be a proper metric space. Two geodesic rays $\gamma_1:[0,\infty)\to X$, $\gamma_2:[0,\infty)\to X$ are {\it equivalent} if there exists a $K\ge 0$ such that for any $t \ge 0$, $|\gamma_1(t)-\gamma_2(t)| \le K$. The {\it Gromov boundary} of $X$, denoted $\partial X$, is the set of equivalence classes of geodesic rays.

If $\gamma:[0,\infty)\to X$ is a geodesic ray in the equivalence class $\xi \in \partial X$, then we say that $\gamma$ limits on $\xi$ and write $\gamma(\infty)=\xi$.

If $(X,d)$ is $\delta$-hyperbolic for some $\delta \ge 0$, then given two points $\xi_1, \xi_2 \in \partial X$, there exists a geodesic $\gamma:(-\infty,+\infty) \to X$ such that the map $t \mapsto \gamma(-t)$ is a geodesic ray limiting on $\xi_1$ and the map $t \mapsto \gamma(t)$ is a geodesic ray limiting on $\xi_2$ [BH99, III.H, lemma 3.2]. In this case, we say that $\gamma$ is a geodesic from $\xi_1$ to $\xi_2$ and write $\gamma(-\infty)=\xi_1, \gamma(+\infty)=\xi_2$.
\end{defn}

\begin{defn}
A {\it generalized ray} is a geodesic $\gamma:I \to X$ where either $I=[0,R]$ for some $R >0$ or $I=[0,\infty)$. In the former case, define $\gamma(t)=\gamma(R)$ for all $t \in [R,\infty]$. We say that $\gamma$ is a generalized ray from $\gamma(0)$ to $\gamma(R)$.
\end{defn}

\begin{defn}
Let $(X,d)$ be a proper Gromov-hyperbolic space with basepoint $p\in X$. We topologize $X \cup \partial X$ as follows. Say that a sequence $\{x_i\}_{i=1}^\infty$ converges to $x_\infty$ if and only if there exists generalized rays $\gamma_i$ from $p$ to $x_i$ such that every subsequence of $\{\gamma_i\}$ has a subsequence that converges uniformly on compact subsets to a geodesic from $p$ to $x_\infty$. 

It is well-known this topology is independent of $p$ and makes $X \cup \partial X$ a compact space in which $\partial X$ is closed (e.g., [BH99, III.H proposition 3.7]).

\end{defn}

\begin{defn}
For any subset $Y \subset X \cup \partial X$, let $L(Y)$ be the intersection of $\partial X$ with the closure of $Y$ in $X \cup \partial X$.
\end{defn}

\begin{defn}[Visual metric]
Let $(X,d)$ be a proper Gromov-hyperbolic space with basepoint $p\in X$. A metric $d_\partial$ on $\partial X$ is called a {\it visual metric} with parameter $a > 0$, if it induces the same topology on $\partial X$ as given above and there exists a constant $C>0$ such that for all $\xi_1,\xi_2 \in \partial X$, if $\gamma$ is a geodesic from $\xi_1$ to $\xi_2$ then
$$C^{-1}a^{-d(p,\gamma)} \le d_\partial(\xi_1,\xi_2) \le C a^{-d(p,\gamma)}.$$
Here, $d(p,\gamma) = \inf_{t} d(p,\gamma(t))$.
\end{defn}

\begin{defn}
A subset $Y \subset X$ is {\it $\epsilon$-quasi-convex} if any geodesic segment $[x,y]$ between points $x,y \in Y$ is contained in the $\epsilon$-neighborhood of $X$. We say that $Y$ is {\it quasi-convex} if it is $\epsilon$-quasiconvex for some $\epsilon>0$.
\end{defn}

\begin{defn}
An infinite group $\Gamma$ acting by isometries on a proper Gromov-hyperbolic space $(X,d)$ is called {\it quasi-convex cocompact} if the action is properly discontinuous, $\Gamma$ does not fix any point of $\partial X$, and for some $\Gamma$-invariant quasi-convex subset $A \subset X$, the quotient $A/\Gamma$ is compact.
\end{defn}

\subsection{Quasi-isometries and quasi-geodesics}

In order to prove theorem \ref{thm:asymptotic}, we show (in the next subsection) that if $p \in X$ and $\epsilon>0$ is sufficiently small, then the map $h \mapsto \phi_\epsilon(h)p$ is a quasi-isometry of $H$ into $X$ (with respect to a fixed word metric on $H$). In this subsection, we introduce the necessary definitions and standard results needed to prove this.

\begin{defn}
Let $(X,d_X), (Y,d_Y)$ be metric spaces, $\lambda \ge 1, c \ge 0$. A map $\pi:X \to Y$ is a {\it $(\lambda, c)$-quasi-isometric embedding} if for all $x,y \in X$,
$$\lambda^{-1}d_X(x,y) - c \le d_Y\big(\pi(x),\pi(y)\big) \le \lambda d_X(x,y) + c.$$
We say that $\pi$ is a {\it quasi-isometric embedding} if it is a $(\lambda,c)$-quasi-isometric embedding for some constants $\lambda \ge 1, c\ge 0$. 
\end{defn}

\begin{defn}
For $\lambda \ge 1, c\ge 0$ a {\it $(\lambda, c)$-quasi-geodesic} in a metric space $(X,d)$ is a $(\lambda,c)$-quasi-isometric embedding $q:I \to X$ where $I$ is an interval on the real line (bounded or unbounded) or else the intersection of $\Z$ with such an interval. It is a {\it quasi-geodesic ray} if $I=[0,\infty)$ or $[0,\infty) \cap \Z$. 
\end{defn}

The theorem below is proven in [BH99, III.H, Theorem 1.7].
\begin{thm}[Stability of Quasi-Geodesics]\label{thm:stability}
For all $\delta \ge 0, \lambda \ge 1$, $c \ge 0$, there exists a constant $R=R(\delta,\lambda,c)$ with the following property.

Let $(X,d)$ denote a proper $\delta$-hyperbolic space. If $q:I \to X$ is a $(\lambda,c)$-quasi-geodesic in $X$ and $[x,y]$ is a geodesic segment joining the endpoints of $q$, then the Hausdorff distance between $[x,y]$ and the image of $q$ is less than $R$.
\end{thm}

\begin{defn}
Let $(X,d)$ be a metric space and $M \ge 0, \lambda \ge 1, c\ge 0$. A path $q:I \to X$ is said to be an {\it $(M,\lambda,c)$-local-quasi-geodesic} if for all $a, b \in I$ with $0\le b-a \le M$, the restriction of $q$ to $[a,b]$ is a $(\lambda,c)$-quasi-geodesic.
\end{defn}

Theorem \ref{thm:localtoglobal} below states that any $(M,\lambda,c)$-local-quasi-geodesic in a $\delta$-hyperbolic space is a $(\lambda',c')$-quasi-geodesic for some $(\lambda',c')$ that depend only on $\delta,\lambda$ and $c$. This fact is well-known (e.g., it is stated in [Gr87, Remark 7.2B]), but I have not found a proof in the literature. The proof given below is based on the proof of theorem 1.13 of [BH99, chapter III.H]. We will need the following lemmas.

\begin{lem}\label{lem:K}
For every $\delta\ge 0, \lambda \ge 1, c\ge 0$, there exists a $K=K(\delta,\lambda,c)\ge 0$ such that the following holds. Let $(X,d)$ be a $\delta$-hyperbolic space. Let $q:I\to X$ be an $(M,\lambda,c)$-local-quasi-geodesic where $I$ is a closed interval of $\R$. Assume $M > 4R\lambda + 8\delta\lambda +2c\lambda+1$ where $R=R(\delta,\lambda,c)$ is as in theorem \ref{thm:stability}. Let $a,b$ be the endpoints of $q(I)$. If $[a,b]$ is a geodesic segment from $a$ to $b$, then $q(I)$ is contained in the $K$-neighborhood of $[a,b]$.
\end{lem}

\begin{proof}
By taking limits of finite intervals, we see that it suffices to consider the case when $I$ is a finite interval. Let $x \in q(I)$ be such that $d(x, [a,b]) = \max_{z\in q(I)} d(z,[a,b])$. 

{\bf Case 1}. Suppose there exists $t\in I$ such that $q(t)=x$ and the interval of length $4R\lambda+8\delta\lambda+2c\lambda+1$ centered at $t$ is not contained in $I$. It follows that
$$d(x,[a,b]) \le \min\Big(d(x,a), d(x,b)\Big) \le 2R\lambda^2+4\delta\lambda^2+c\lambda^2+ \lambda+ c.$$

{\bf Case 2}. Assume now, that there is a subinterval $I' \subset I$ centered at $x$ of length $N$ for some $N$ with $M\ge N > 4R\lambda + 8\delta\lambda +2c\lambda$. Let $y$ and $z$ be the endpoints of $q(I')$. Let $x' \in [y,z]$ be the closest point to $x$ on $[y,z]$. By the previous theorem, $d(x,x') \le R$.

Let $y'$ be the closest point to $y$ on $[a,b]$. Let $z'$ be the closest point to $z$ on $[a,b]$. By $\delta$-hyperbolicity, there exists a point $w$ in $ [z,z'] \cup [y,z]$ with $d(w,x')\le \delta$. Suppose, for a contradiction, that $w \in [z,z']$. By choice of $x$, 
$$d(z,w)+d(w,z')=d(z,z') \le d(x,z') \le d(x, x') + d(x',w) + d(w,z').$$
Hence,
$$d(z,w) \le d(x,x') + d(x',w) \le R + \delta.$$
By the triangle inequality,
$$d(z,w) \ge d(z,x) - d(x,w) \ge \lambda^{-1}N/2 -c - R - \delta.$$
Hence we obtain $N \le 4R\lambda + 4\delta\lambda +2c\lambda$, a contradiction. Thus, $w \in [y,z]$.

By $\delta$-hyperbolicity, there exists a point $w' \in [y,y'] \cup [y',z']$ with $d(w,w') \le \delta$. Arguing as above, we can show that $w' \in [y',z'] \subset [a,b]$. Hence
$$d(x, [a,b]) \le d(x,w') \le d(x,x') + d(x',w) + d(w,w') \le R + 2\delta.$$
Let $K=\max\Big(R+2\delta, 2R\lambda^2+4\delta\lambda^2+c\lambda^2+ \lambda+ c)$ to finish the proposition.
\end{proof}

\begin{lem}\label{lem:monotone}[Monotonicity]
For any $\delta \ge 0, \lambda \ge 1$ and $c\ge 0$, there exists an $M_0=M_0(\delta,\lambda,c) \ge 0$ such that the following holds. Let $(X,d)$ be a $\delta$-hyperbolic space. Let $q:I\to X$ be any $(M,\lambda,c)$-local-quasi-geodesic for some $M\ge M_0$. Let $R=R(\delta,\lambda,c)$ and $K=K(\delta,\lambda,c)$ be as in theorem \ref{thm:stability} and lemma \ref{lem:K}. Let $a,b$ be the endpoints if $q(I)$. 

Let $I' \subset I$ be a finite subinterval. Assume the length of $I'$, $l(I')$, satisfies  $M \ge l(I') \ge M_0$. Let $x, y$ be the endpoints of $q(I')$. Let $m=q(m_*)$ where $m_*$ is the midpoint of $I'$.

 Let $x',y',m'$ be points on a geodesic $[a,b]$ from $a$ to $b$ such that $d(x',x) \le K$, $d(y',y)\le K$ and $d(m',m)\le K$. Then $m'$ lies between $x'$ and $y'$. I.e., $m'$ is in the subarc of $[a,b]$ between $x'$ and $y'$.
\end{lem}

\begin{proof}
Let $M_0 = 4\lambda K+4R\lambda+8\delta\lambda+2c\lambda + 2$. By theorem \ref{thm:stability}, there exists a point $m_0$ on $[x,y]$ with $d(m_0,m) \le R$.

Let $x_0$ be a point on $[x,y]$ with $d(x,x_0) = \delta$. Similarly, let $y_0$ be a point on $[x,y]$ with $d(y,y_0) = \delta$. By $\delta$-hyperbolicity, the geodesic triangle $\Delta(x,x_0,x')$ is contained in the $(K + 2\delta)$-neighborhood of $x$. Similarly, the triangle $\Delta(y,y_0,y')$ is contained in the $(K + 2\delta)$-neighborhood of $y$.

By $\delta$-hyperbolicity, there exists a point $m'_0$ in $[x_0, y'] \cup [y_0,y']$ with $d(m'_0,m_0)\le \delta$. We claim that $m'_0 \in [x_0,y']$. To see this, note that
$$d(m_0, y) \ge d(m,y)-d(m_0,m) \ge \frac{l(I')}{2\lambda} - c - R > K+ 3\delta.$$
If $m'_0 \in [y_0,y']$ then $d(m_0,y) \le d(m_0,m'_0) + d(m'_0,y) \le K+ 3\delta$, contradicting the above. So $m'_0 \in [x_0,y']$.

By $\delta$-hyperbolicity, there exists a point $m'' \in [x_0,x'] \cup [x',y']$ with $d(m'', m'_0) \le \delta$. Arguing as above, we see that $m'' \in [x',y']$. If $m'''$ is any point in $[m'',m']$ then by $\delta$-hyperbolicity (applied to the triangle $\Delta(m_0,m',m'')$), it follows that $d(m''',m_0) \le K+3\delta $.

Since 
$$d(x',m_0) \ge d(x,m) - d(m,m_0) - d(x',x) \ge \frac{l(I')}{2\lambda} - c - R - K > K + 3\delta,$$
it follows that $x' \notin [m'',m']$. Similarly, $y' \notin [m'',m']$. This proves that $m' \in [x',y']$ as claimed.

\end{proof}

\begin{thm}\label{thm:localtoglobal}
For any $\delta \ge 0, \lambda > 1$ and $c\ge 0$, there exists $M_1\ge 0$, $\lambda'=\lambda' \ge 1$ and $c'\ge 0$  such that the following holds. Let $(X,d)$ be a $\delta$-hyperbolic space. Let $q:I\to X$ be an $(M,\lambda,c)$-local-quasi-geodesic for some $M \ge M_1$. Then $q$ is a $(\lambda',2K)$-quasi-geodesic where $K$ is as in lemma \ref{lem:K}.
\end{thm}

\begin{proof}
Let $R=R(\delta,\lambda,c)$, $K=K(\delta,\lambda,c)$, $M_0=M_0(\delta,\lambda,c)$ be as in theorem \ref{thm:stability}, lemma \ref{lem:K} and the previous lemma. Choose $M_1$ so that $M_1 \ge M_0$ and $\lambda^{-1} - \frac{2c+4K}{M_1} >0$. Choose $\lambda'$ so that $\frac{1}{\lambda'} \le \lambda^{-1} - \frac{2c+4K}{M_1}$ and $\lambda ' \ge \lambda + \frac{2c}{M_1}$. Let $c'=4K+2$.

 Let $a,b \in I$ be such $b>a$. Let $a=x_0<x_1<\ldots <x_n = b$ be a subdivision into $n \le \frac{2(b-a)}{M_1} + 1$ subintervals, each of length at most $\frac{M}{2}$. By lemma \ref{lem:K}, for each $i$ there exists a point $x'_i \in [q(a),q(b)]$ with $d(x'_i, x_i) \le K$. By the previous lemma, $x'_i \in [x'_{i-1},x'_{i+1}]$ for all $n-1\ge i\ge 1$. Thus,
\begin{eqnarray*}
d\big(q(a),q(b)\big) &\ge& d(x'_0,x'_n) - 2K= -2K + \sum_{i=1}^n d(x'_i,x'_{i-1})\\
&\ge& -2K +  \sum_{i=1}^n d\big(q(x_i),q(x_{i-1})\big) - 2K\ge -2K +  \sum_{i=1}^n \lambda^{-1}|x_i-x_{i-1}| -c - 2K\\
&=& \lambda^{-1}|b-a| - cn - 2K(n+1)\ge |b-a|\Big( \lambda^{-1} - \frac{2c+4K}{M_1}\Big) - 4K -c\\
&\ge& \frac{|b-a|}{\lambda'} - c'.
\end{eqnarray*}
On the other hand,
\begin{eqnarray*}
d\big(q(a),q(b)\big) &\le& \sum_{i=1}^n d\big(q(x_i),q(x_{i-1})\big)\le \sum_{i=1}^n\lambda |x_i-x_{i-1}| +c = \lambda|a-b| + cn\\
&\le& |a-b|\Big(\lambda + \frac{2c}{M_1}\Big) + c \le \lambda' |a-b| + c'.
\end{eqnarray*}
 \end{proof}

\begin{defn}
Let $X, Y$ be metric spaces and $0 \le M, \lambda \ge 1, c\ge 0$. A map $q:Y \to X$ is said to be an {\it $(M,\lambda,c)$-local-quasi-isometric} embedding if for all $y \in Y$, $q$ restricted to the ball of radius $M$ centered at $y$ is a $(\lambda,c)$-quasi-isometric embedding.
\end{defn}

\begin{cor}\label{cor:localqi}
Let $M_1,\lambda', c'$ be as in lemma \ref{lem:K} and the previous theorem. Let $(X,d)$ be a $\delta$-hyperbolic space and $(Y,d_Y)$ a geodesic space. If $M \ge M_1$ and $q:Y \to X$ is an $(M,\lambda,c)$-local-quasi-isometric embedding then $q$ is a $(\lambda',c')$-quasi-isometric embedding.
\end{cor}

\begin{proof}
If $\gamma:I \to Y$ is a geodesic then $q \circ \gamma$ is an $(M,\lambda,c)$-local-quasi-geodesic. The theorem above implies that $q\circ \gamma$ is a $(\lambda',c')$-quasi-geodesic. Since this is true for all $\gamma$, $q$ is a $(\lambda',c')$-quasi-isometry.
\end{proof}

\subsection{Perturbations and quasi-isometric embeddings}

In this subsection, we take the first step in proving theorem \ref{thm:asymptotic} by showing that if $\epsilon>0$ is sufficiently small, then $\phi_\epsilon$ is a quasi-isometric embedding. 

\begin{defn}
Let $H$ be an abstract group with finite symmetric generating set $S$. The {\it Cayley graph} $\sC$ of $H$ induced by $S$ is the graph with vertex set $H$ and so that for every $h\in H$ and $s\in S$ there is an edge from $h$ to $hs$. Let each edge be isometric with the unit interval and let $d_S$ denote the resulting path metric. This makes $\sC$ a geodesic space. 
\end{defn}
From now on, fix $H, S$ as above. Let $(X,d)$ be a $\delta$-hyperbolic space. Let $\phi:H \to \Isom(X)$ be an injective homomorphism onto a quasi-convex cocompact subgroup. Let $d_{\Isom(X)}$ be a left-invariant metric on $\Isom(X)$ inducing the topology of uniform convergence on compact sets.


\begin{lem}\label{lem:Hqi}
Let $p\in X$. Define $\pi_p:H \to X$ by $\pi_p(h) = \phi(h)p$. Then $\pi_p$ is a quasi-isometric embedding.
\end{lem}

\begin{proof}
Let $A \subset X$ be a $\phi(H)$-invariant quasi-convex subset such that $A/\phi(H)$ is compact. Let $a \in A$. By the \v Svarc-Milnor lemma, the map $h \mapsto \phi(h)a$ is a $(\lambda,c)$-quasi-isometric embedding for some $\lambda \ge 1, c \ge 0$. This lemma was discovered in the fifties [Ef53, Sv55] and rediscovered by Milnor [Mi68, lemma 2]. It is also proven in [BH99]. 

For any $h,g \in H$,
\begin{eqnarray*}
\Big|d\big( \phi(h)p, \phi(g)p \big) - d\big( \phi(h)a, \phi(g)a \big)\Big| \le d\big( \phi(h)p, \phi(h)a \big) + d\big( \phi(g)a, \phi(g)p \big)=2d(a,p).
\end{eqnarray*}
Since $\lambda^{-1} d_S(h,g) - c \le  d\big( \phi(h)a, \phi(g)a \big) \le \lambda d_S(h,g) + c$, this implies
$$\lambda^{-1} d_S(h,g) - c - 2d(a,p)\le  d\big( \phi(h)p, \phi(g)p \big) \le \lambda d_S(h,g) + c +2d(a,p).$$
Hence $h \mapsto \phi(h)p$ is a $(\lambda, c+ 2d(p,a))$-quasi-isometric embedding.
\end{proof}

\begin{lem}\label{lem:local}
For any $N, \sigma> 0$ and any $p \in X$, there exists an $\epsilon_0>0$ such that the following holds. If $0\le\epsilon\le\epsilon_0$ and $\phi_\epsilon:H \to G$ is an $\epsilon$-perturbation of $\phi$ then for all $g,h \in H$ with $d_S(g,h)\le N$, 
$$\Big|d\big( \phi(g)p, \phi(h)p\big) - d\big(\phi_\epsilon(g)p, \phi_\epsilon(h)p\big) \Big| \le \sigma.$$
\end{lem}

\begin{proof}
Let $\epsilon_0>0$ be such that if $m \le N$, $s_1,\ldots,s_m \in S$ and $s'_1,\ldots,s'_m \in \Isom(X)$ are such that $d_{\Isom(X)}(s'_i , \phi(s_i)) \le \epsilon_0$ then
$$\Big|d(s'_1 \cdots s'_m p, p) -d\big(\phi(s_1)\cdots \phi(s_m)p,p\big)\Big| \le \sigma.$$

Let $0\le \epsilon \le \epsilon_0$ and let $\phi_\epsilon$ be an $\epsilon$-perturbation of $\phi$. Let $g,h \in H$ with $d_S(g,h) \le N$. So $g=hs_1s_2\cdots s_m$ for some $s_i \in S$ and $m\le N$. Because $\phi_\epsilon$ is an $\epsilon$-perturbation of $\phi$, there exist elements $s'_{i}\in G$ with $d_{Isom(X)}(s'_{i},\phi(s_i)) \le \epsilon$ and $\phi_\epsilon(g)=\phi_\epsilon(h)s'_{1}\cdots s'_{m}$. Thus
$$\Big|d\big( \phi(g)p, \phi(h)p\big) - d\big(\phi_\epsilon(g)p, \phi_\epsilon(h)p\big) \Big| = \Big|d\big(\phi(s_1)\cdots \phi(s_m)p,p\big)- d\big(s'_{1} \cdots s'_{m} p, p\big)\Big| \le \sigma.$$

\end{proof}

\begin{prop}\label{prop:qi}
Let $p\in X$. There exists an $\epsilon_0>0, \lambda'\ge 1$ and $c'\ge 0$ such that if $0 \le \epsilon \le \epsilon_0$ and $\phi_\epsilon:H \to \Isom(X)$ is any $\epsilon$-perturbation of $\phi$, then the map $h \mapsto \phi_\epsilon(h)p$ is a $(\lambda',c')$-quasi-isometric embedding of $H$ into $X$.
\end{prop}

\begin{proof}
By lemma \ref{lem:Hqi}, the map $h \mapsto \phi(h)p$ is a $(\lambda,c)$-quasi-isometric embedding for some $\lambda \ge 1, c\ge 0$. Let $\sigma \ge 0$. Let $M \ge M_1$ where $M_1=M_1(\delta,\lambda+\sigma,c+\sigma+2\lambda+2)$ is as defined in corollary \ref{cor:localqi}.

By the previous lemma, it follows that there exists an $\epsilon_0$ such that if $0\le \epsilon \le \epsilon_0$ and $\phi_\epsilon:H \to \Isom(X)$ is an $\epsilon$-perturbation of $\phi$, then the map $h \mapsto \phi_\epsilon(h)p$ is a $(M,\lambda+\sigma,c+\sigma)$-local-quasi-isometric embedding. We can extend this map to the Cayley graph $\sC$. For example, for each edge we could choose an endpoint and map the entire edge to the image of that endpoint. The resulting map is an $(M,\lambda+\sigma,c+\sigma+2\lambda +2)$-local-quasi-isometric embedding. By corollary \ref{cor:localqi}, $\phi_\epsilon$ is a $(\lambda',c')$-quasi-isometry for some constants $\lambda'\ge 1, c'\ge 0$.
\end{proof}

\subsection{Bi-Lipschitz maps}




Here we conclude that the map $\phi(h)p \mapsto \phi_\epsilon(h)p$ is bi-Lipschitz with constant that tends to $1$ as $\epsilon$ tends to $0$. This result is the key ingredient to proving theorem \ref{thm:asymptotic}. We need the next lemma.

\begin{lem}\label{lem:rho}
Let $p\in X$. Then there exists a $\rho>0$ such that if $g,h \in H$ and $g \ne h$ then $d\big( \phi(g)p, \phi(h)p \big) \ge \rho$.
\end{lem}

\begin{proof}
By lemma \ref{lem:Hqi}, there exists $\lambda \ge 1$ and $c \ge 0$ such that the map $f \mapsto \phi(f)p$ is a $(\lambda,c)$-quasi-isometric embedding of $H$ into $X$. Let $N$ be an integer such that $\lambda^{-1}N - c \ge 1$. Let 
$$\rho_0 = \min \big \{d(p,\phi(g)p) ~|~ d_S(g,id)\le N, ~g \ne id\big\}.$$  
Let $g,h \in H$ with $g\ne h$. If $d_S(g,h) \le N$ then $d\big( \phi(g)p, \phi(h)p \big) =  d\big( \phi(h^{-1}g)p, p \big) \ge \rho_0.$ If $d_S(g,h) >N$ then, since the map $f \mapsto \phi(f)p$ is a $(\lambda,c)$-quasi-isometry,
 $d\big( \phi(g)p, \phi(h)p \big) \ge \lambda^{-1}N - c \ge 1.$ Set $\rho=\min(1,\rho_0)$ to finish the lemma.
\end{proof}

\begin{prop}\label{prop:bilip}
Let $p\in X$. Let $\kappa >1$. Then there exists an $\epsilon_0 >0$ such that if $0\le \epsilon \le \epsilon_0$ and $\phi_\epsilon:H \to \Isom(X)$ is an $\epsilon$-perturbation of $\phi$, then for all $g,h \in H$,
$$\kappa^{-1} d\big( \phi(h)p, \phi(g)p \big) \le d\big( \phi_\epsilon(h)p, \phi_\epsilon(g)p \big) \le \kappa d\big( \phi(h)p, \phi(g)p \big).$$
\end{prop}

\begin{proof}

By proposition \ref{prop:qi}, there exists an $\epsilon_1>0, \lambda' \ge 1, c'\ge 0$ such that if $0\le \epsilon\le\epsilon_1$ then map $h \mapsto \phi_\epsilon(h)p$ is a $(\lambda',c')$-quasi-isometric embedding. 


Let $\kappa'$ be such that $1<\kappa' < \kappa$.
Let $N,\sigma>0$ be such that $1+ \frac{\sigma \lambda'}{N} < \kappa'$ and $\sigma\Big(1 + \frac{c}{N}\Big) \frac{\kappa\kappa'}{\kappa-\kappa'} <\rho$ where $\rho$ is as defined in the previous lemma.

 By lemma \ref{lem:local}, there exists $\epsilon_0$ with $0<\epsilon_0 \le \epsilon_1$ such that if $g,h \in H$ are such that $d(g,h)\le N$ and $0\le \epsilon \le \epsilon_0$, then
$$\Big| d\big(\phi_\epsilon(g)p, \phi_\epsilon(h)p\big) - d\big(\phi(g)p, \phi(h)p\big)\Big| \le \sigma.$$



Assume now that $0\le \epsilon \le \epsilon_0$. Let $g,h \in H$. Let $q:I \to \sC$ be a geodesic from $g$ to $h$. Let $t_0<t_1<\ldots<t_n$ be points in $I$ such that $q(t_0)=g, q(t_n)=h$, $d_S(t_i,t_{i+1}) \le N$ (for all $i$) and $n\le 1+\frac{d_S(g,h)}{N}$. Then
\begin{eqnarray*}
&&\Big| d\big(\phi_\epsilon(g)p, \phi_\epsilon(h)p\big) - d\big(\phi(g)p, \phi(h)p\big)\Big|\\
&\le& \sum_{i=1}^n \Big| d\big(\phi_\epsilon(q(t_i))p, \phi_\epsilon(q(t_{i-1}))p\big) - d\big(\phi(q(t_i))p, \phi(q(t_{i-1}))p\big)\Big| \le n\sigma\le \sigma+\frac{\sigma d_S(g,h)}{N}.
\end{eqnarray*}
Since $\phi$ is a $(\lambda',c')$-quasi-isometric embedding, this implies
\begin{eqnarray*}
 d\big(\phi_\epsilon(g)p, \phi_\epsilon(h)p\big) - d\big(\phi(g)p, \phi(h)p\big)&\le&  \sigma+\frac{\sigma \lambda'd\big(\phi(g)p, \phi(h)p\big) }{N} + \frac{c'\sigma}{N}.
\end{eqnarray*}
This simplifies to:
\begin{eqnarray*}
 d\big(\phi_\epsilon(g)p, \phi_\epsilon(h)p\big) &\le& \Big(1+ \frac{\sigma \lambda'}{N}\Big)d\big(\phi(g)p, \phi(h)p\big) + \sigma\Big(1+ \frac{c}{N}\Big).
\end{eqnarray*}

By reversing the roles of $\phi_\epsilon$ and $\phi$, we obtain
\begin{eqnarray*}
 d\big(\phi_\epsilon(g)p, \phi_\epsilon(h)p\big) &\ge& \Big(1+ \frac{\sigma \lambda'}{N}\Big)^{-1}d\big(\phi(g)p, \phi(h)p\big) - \Big(1+ \frac{\sigma \lambda'}{N}\Big)^{-1}\sigma\Big(1+ \frac{c}{N}\Big).
\end{eqnarray*}
Thus if $C=\sigma\Big(1+ \frac{c}{N}\Big)$,
\begin{eqnarray}\label{eqn:bilip1}
\frac{1}{\kappa'}d\big(\phi(g)p, \phi(h)p\big) -C  \le& d\big(\phi_\epsilon(g)p, \phi_\epsilon(h)p\big) &\le \kappa' d\big(\phi(g)p, \phi(h)p\big) +C.
\end{eqnarray}
The proposition is automatically true if $g=h$. By the previous lemma, if $g \ne h$ then $ d\big(\phi(g)p, \phi(h)p\big) \ge \rho$. So,
\begin{eqnarray}\label{eqn:bilip2}
\kappa'  d\big(\phi(g)p, \phi(h)p\big) +\sigma\Big(1+ \frac{c}{N}\Big) &\le& \kappa d\big(\phi(g)p, \phi(h)p\big) -(\kappa -\kappa')\rho + \sigma\Big(1+ \frac{c}{N}\Big)\\
&\le& \kappa d\big(\phi(g)p, \phi(h)p\big).
\end{eqnarray}
Similarly,
\begin{eqnarray}\label{eqn:bilip3}
\frac{1}{\kappa'}d\big(\phi(g)p, \phi(h)p\big) -\sigma\Big(1+ \frac{c}{N}\Big) \ge \kappa^{-1}d\big(\phi(g)p, \phi(h)p\big).
\end{eqnarray}
The inequalities (\ref{eqn:bilip1} - \ref{eqn:bilip3}) imply the proposition.

\end{proof}

\subsection{Proof of theorem \ref{thm:asymptotic}}

We will prove each item of theorem \ref{thm:asymptotic} separately. 

\begin{prop}\label{prop:1-1}
There exists an $\epsilon_0 >0$ such that if $0 \le \epsilon \le \epsilon_0$ and $\phi_\epsilon:H \to G$ is an $\epsilon$-perturbation of $\phi$ then $\phi_\epsilon$ is 1-1.
\end{prop}

\begin{proof}
Let $p\in X$. Let $\rho$ be as in lemma \ref{lem:rho}. Let $\kappa>1$. By proposition \ref{prop:bilip} there exists an $\epsilon_0 > 0$ such that if $0\le \epsilon\le\epsilon_0$ then for all $g,h \in H$ with $g\ne h$,
$$0<\kappa^{-1}\rho \le \kappa^{-1} d\big( \phi(h)p, \phi(g)p \big) \le d\big( \phi_\epsilon(h)p, \phi_\epsilon(g)p \big).$$
\end{proof}

Fix a visual metric $d_\partial$ on $\partial X$. The next lemma is well-known. For example, it is an immediate consequence of lemma 3.6 in [BH99, chapter III.H].
\begin{lem}
Let $p\in X$, $W>0$ and $C >0$. Then there exists a $N\ge 0$ such that if $q_1:[0,\infty) \to X$ and $q_2:[0,\infty) \to X$ are geodesic rays with $q_1(0)=q_2(0)=p$ and $d\big(q_1(N),q_2(N)\big) < W$ then $d_\partial(q_1(\infty),q_2(\infty)) <C$.
\end{lem}

Because the map $h \mapsto \phi(h)p$ is a quasi-isometric embedding into a $\delta$-hyperbolic space, it follows that the Cayley graph $\sC$ is $\delta$-hyperbolic. Let $\partial H = \partial \sC$.

\begin{prop}\label{prop:limit set}
For all $C>0$ and $p\in X$, there exists an $\epsilon_0>0$ such that if $0 \le \epsilon \le \epsilon_0$ and $\phi_\epsilon:H \to G$ is an $\epsilon$-perturbation of $\phi$ then $d_{\Haus}(L(\phi(H)p), L(\phi_\epsilon(H)p)) < C$.
\end{prop}

\begin{proof}
By proposition \ref{prop:qi}, there exists $\lambda'\ge 1, c'\ge 0, \epsilon_2 > 0$ such that if $0 \le \epsilon \le \epsilon_2$ then the map $h \mapsto \phi_\epsilon(h)p$ is a $(\lambda',c')$-quasi-isometry.

Let $C, \sigma >0$. Let $W=4R+2\sigma$ where $R=R(\delta,\lambda',c')$ is as in theorem \ref{thm:stability}. Let $N$ be as in the previous lemma.

By lemma \ref{lem:local}, there exists an $\epsilon_1 >0$ such that if $0 \le \epsilon \le \epsilon_1$ then for all $h \in H$ with $d_S(h,id)\le \lambda'(N+R) + \lambda'c'$, $d(\phi_\epsilon(h)p, \phi(h)p) \le \sigma.$

Now let $\epsilon_0=\min(\epsilon_1,\epsilon_2)$. Let $0\le \epsilon\le \epsilon_0$. Let $\phi_\epsilon:H \to \Isom(X)$ be any $\epsilon$-perturbation of $\phi$. Let $\pi_\epsilon:H \to X$ be the map $\pi_\epsilon(h) = \phi_\epsilon(h)p$. Extend $\phi_\epsilon$ to all of $\sC$ by choosing, for each edge in $\sC$, one of its endpoints and mapping the entire edge to the image of its endpoint. The resulting map is still a quasi-isometric embedding. By [BH99, III.H, theorem 3.9], $\pi_\epsilon$ has a unique continuous extension $\pi_\epsilon:\sC \cup \partial \sC \to X \cup \partial X$ that restricts to a topological embedding of $\partial \sC$ into $\partial X$ with image equal to $L(\phi_\epsilon(H)p)$.

Similarly, let $\pi_0:H \to X$ be the map $\pi_0(h)=\phi(h)p$. Extend it to a map $\pi_0:\sC \cup \partial \sC \to X \cup \partial X$ in a similar manner.

Let $\xi \in \partial H$. Let $q:[0,\infty) \to \sC$ be a geodesic ray with $q(0)=id$ and $q(\infty)=\xi$. By definition $\pi_\epsilon \circ q$ is a $(\lambda',c')$-quasi-geodesic. Let $\zeta_\epsilon = \pi_\epsilon(\xi)$. Let $\gamma_\epsilon:[0,\infty) \to X$ be a geodesic ray with $\gamma_\epsilon(0)=p$ and $\gamma_\epsilon(\infty)=\zeta_\epsilon$.

 By theorem \ref{thm:stability}, the Hausdorff distance between $\pi_\epsilon\circ q([0,\infty)\cap \Z) $ and $\gamma_\epsilon([0,\infty))$ is at most $R$. So there exists an $h \in H$ with $d(\gamma_\epsilon(N),\pi_\epsilon(h)) \le R$. Since $\pi_\epsilon$ is a $(\lambda',c')$-quasi-isometry,
$$d_S(h,id) \le \lambda' d(\pi_\epsilon(h), p ) + \lambda' c' \le \lambda'(N+R) +\lambda'c'.$$
Thus, $d(\pi_\epsilon(h), \pi_0(h) ) \le \sigma$. By theorem \ref{thm:stability}, there exists a $t \ge 0$ with $d(\pi_0(h), \gamma_0(t))\le R$. Observe that
\begin{eqnarray*}
|t-N|&=&\Big|d\big( \gamma_0(t), p \big)-d\big( \gamma_\epsilon(N), p \big)\Big|  \le d\big(\gamma_0(t), \gamma_\epsilon(N)\big)\\
&\le& d\big( \gamma_0(t), \pi_0(h)\big) + d\big( \pi_0(h), \pi_\epsilon(h) \big) + d\big( \pi_\epsilon(h), \gamma_\epsilon(N) \big) \le 2R+\sigma.
\end{eqnarray*}
This implies
\begin{eqnarray*}
\Big| d\big( \gamma_0(N), p \big)-d\big( \gamma_\epsilon(N), p \big)\Big| &\le&\Big| d\big( \gamma_0(N), p \big)-d\big( \gamma_0(t), p \big)\Big| + \Big| d\big( \gamma_0(t), p \big)-d\big( \gamma_\epsilon(N), p \big)\Big|\\
&\le& 4R+2\sigma.
\end{eqnarray*}
By the choice of $N$, this implies that $d_\partial\big(\pi_\epsilon(\xi),\pi_0(\xi)\big) \le C$. Since this is true for all $\xi \in \partial H$, it follows that $d_{\Haus}( \pi_\epsilon(\partial H), \pi_0(\partial H)) \le C$ as claimed.
\end{proof}

To prove the last part of theorem \ref{thm:asymptotic}, we rely on a well-known generalization of Patterson-Sullivan theory to word hyperbolic groups due to Coornaert. This is explained next.

\begin{defn}\label{defn:exponent}
Let $\Gamma \subset \Isom(X)$ be a discrete subset, $p, q\in X$ and $s> 0$. Then the {\it Poincar\'e series} of $\Gamma$ with respect to the visual parameter $a>0$ is defined by
$$g_s(p,q) = \sum_{\gamma \in \Gamma} a^{-sd(p,\gamma q)}.$$
A short calculation shows that if $g_s(p,q)$ is finite for some pair $(p,q)$ then it is
finite for all such pairs. So let $\delta_a(\Gamma)=\inf\{s: \, g_s(p,q)<\infty\}$ be the {\it exponent of convergence} of $\Gamma$ with respect to the parameter $a$.
In [Co93] it was proven that if $\Gamma$ is a quasi-convex cocompact subgroup then $\delta_a(\Gamma)=\HD(L(\Gamma))$, the Hausdorff-dimension of the limit set of $\Gamma$ with respect to a visual metric $d_\partial$ with parameter $a$.
\end{defn}

In order to apply Coornaert's result, we need the next lemma.
\begin{lem}
There exists an $\epsilon_0>0$ such that if $0 \le \epsilon \le \epsilon_0$, $\phi_\epsilon:H \to \Isom(X)$ is an $\epsilon$-perturbation of $\phi$ and $\phi_\epsilon$ restricted to $H'$ is a homomorphism (for some $H'<H$ with finite index) then $\phi_\epsilon(H')$ is quasi-convex cocompact.
\end{lem}

\begin{proof}
Let $p\in X$. By proposition \ref{prop:qi}, there exists $\lambda'\ge 1, c'\ge 0, \epsilon_2 > 0$ such that if $0 \le \epsilon \le \epsilon_2$ then the map $h \mapsto \phi_\epsilon(h)p$ is a $(\lambda',c')$-quasi-isometry. Thus if $q:I \cap \Z \to H$ is any map with $d_S(q(a),q(b)) = |a-b|$ for $a,b \in I\cap \Z$, then $t \mapsto \phi_\epsilon(q(t))p$ is a $(\lambda',c')$-quasi-geodesic. This implies, by theorem \ref{thm:stability}, that if $A = \{\phi_\epsilon(h)p ~|~ h\in H'\}$, then $A$ is quasi-convex. Of course, it is $\phi_\epsilon(H')$-invariant. Since $A/\phi_\epsilon(H')$ is a single point, it is compact.
\end{proof}

\begin{prop}\label{prop:HD}
Let $C>0$. There exists an $\epsilon_0>0$ such that if $0 \le \epsilon \le \epsilon_0$ and $\phi_\epsilon:H \to \Isom(X)$ is an $\epsilon$-perturbation of $\phi$ that is virtually a homomorphism then $\Big| \HD(L(\phi_\epsilon) - \HD(L(\phi)) \Big| \le C$.
\end{prop}

\begin{proof}
It follows from proposition \ref{prop:bilip} that there exists an $\epsilon_0>0$ such that if $0 \le \epsilon \le \epsilon_0$ and $\phi_\epsilon:H \to \Isom(X)$ is an $\epsilon$-perturbation of $\phi$ then $\big|\delta(\phi(H)) - \delta(\phi_\epsilon(H))\big| < C$. 

Assume $\phi_\epsilon$ is virtually a homomorphism. So there exists a finite-index subgroup $H'<H$ such that $\phi_\epsilon$ restricted to $H'$ is a homomorphism. By the previous lemma and Coornaert's result (mentioned in definition \ref{defn:exponent}), $\delta_a(\phi_\epsilon(H')) = \HD(L(\phi_\epsilon(H')p))$. So it suffices to show that $\delta_a(\phi_\epsilon(H'))=\delta_a(\phi_\epsilon(H))$ and $L(\phi_\epsilon(H')p) = L(\phi_\epsilon(H)p)$. These are easy exercises left to the reader. 
\end{proof}

Theorem \ref{thm:asymptotic} now follows from propositions \ref{prop:1-1}, \ref{prop:limit set} and \ref{prop:HD}.

\section{Surface Groups and Aperiodic Tilings}\label{section:surface}

Because of the surface subgroup conjecture, we would like to generalize theorem \ref{thm:uniform} to allow $F$ to be a surface group. The only step in the proof which requires $F$ to be a free group is in theorem \ref{thm:core}: showing the existence of a periodic point in a certain graph subshift over $F$. This points to a general problem: find conditions on graph subshifts over a surface group that guarantee the existence of a periodic point. We will consider a continuous version of this problem; replacing the surface group with $\PSL_2(\R)=\Isom^+(\H^2)$ and show, by an explicit counterexample, that the existence of an invariant Borel probability measure is not sufficient. However, up to minor variations, this is the only known counterexample. We make a precise conjecture to the effect that this counterexample is unique and show that it implies the surface subgroup conjecture.


To begin, let us define tiling spaces, which are the continuous analog of graph subshifts.


\begin{defn}

A {\it tile} is a curvilinear polygon in $\H^2$. We think of it as
a compact subset $\H^2$ (equal to the closure of its interior) and also as a finite CW-complex isomorphic to a
polygon. If $P=\{\tau_1, \tau_2,...\}$ is a set of tiles, then a {\it tiling by $P$} is a
collection $T$ of congruent copies of the tiles in $P$ such that
\begin{itemize}
\item (covering) the union of all tiles in $T$ equals the whole plane and
\item (edge-to-edge) for any distinct pair $\tau_1, \tau_2$ of tiles in $T$ the intersection of $\tau_1$ with
$\tau_2$ is either empty, a vertex of both, or an edge of both.
\end{itemize}

Let $\sT(P)$ be the set of all tilings by $P$. It has the following topology. If
$\{T_i\}$ is a sequence of tilings then $T_i$ converges to $T_\infty$ if for every tile $\tau \in T_\infty$, there exist tiles $\tau_i \in T_i$ with $\tau_i$ converging
to $\tau$ in the Hausdorff topology on closed subsets of the plane.

$\Isom^+(\H^2)$ acts on $\sT(P)$ in the obvious way: $\forall g\in \Isom^+(\H^2)$, $gT=\{g\tau \, | \, \tau \in T\}$. The {\it symmetry group} of
$T$ is the group of all isometries that fix $T$. $T$ is {\it periodic} if its symmetry
group is cofinite. If $P$ is finite, then the symmetry group of $T$ is necessarily discrete.

If the set of tiles $P$ is such that $\sT(P)$ is nonempty and {\it every} tiling $T \in
\sT(P)$ is nonperiodic then $P$ is {\it aperiodic}.

\end{defn}

\begin{thm}
There exists an aperiodic set of tiles $P=\{\tau,\sigma\}$ such that $\sT(P)$ is nonempty and there is a $\Isom^+(\H^2)$-invariant Borel probability measure on $\sT(P)$.
\end{thm}

\begin{proof}[Proof Sketch]
This example first appeared in [BHRS05]. The starting point is a slight modification of an aperiodic tile set described in [Pe78]. Identify $\H^2$ with the upperhalf plane model. So $\H^2=\{x+iy \in \C \, |\, y >0\}$
with the metric $ds^2 = \frac{dx^2 + dy^2}{y^2}$. Euclidean similarities that preserve
$\H^2$ are isometries of the hyperbolic metric.

For $w>0$, let $\sigma=\sigma_w$ be the Euclidean rectangle with vertices $i, i+w, 2i, 2i+w$. As a CW-complex, it is to be regarded as a pentagon where the extra vertex is at $i+w/2$.

\begin{figure}[htb]
\begin{center}
 \psfig{file=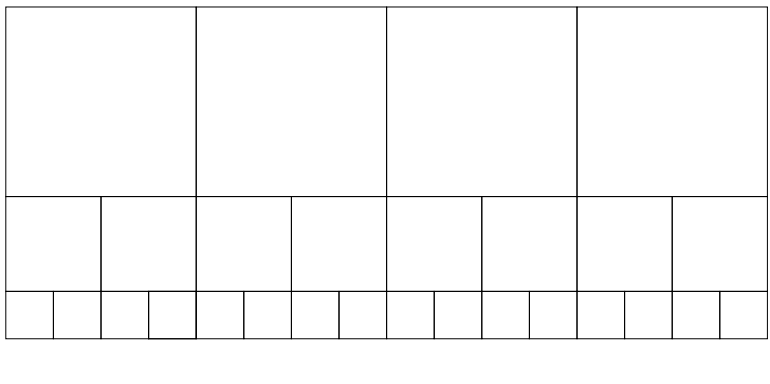,height=2in,width=4in}
\caption{Part of a tiling by $\{\sigma\}$ in the upper-half plane model.}
\label{fig:binarytiling}
\end{center}
\end{figure}

Let $t,s$ be the isometries of $\H^2$ given by $t(z)=z+w$ and $s(z)=2z$. Then $T=\{s^n t^m \sigma_w|\, n, m \in \Z\}$
is a tiling by $P$. See figure \ref{fig:binarytiling}. Thus $\sT(\{\sigma\})$ is nonempty.

 However, it can be shown that $\sT(\{\sigma\})$ does not admit any $\PSL_2(\R)$-invariant Borel probability measures. This is because there is a $\PSL_2(\R)$-equivariant map from $\sT(\{\sigma\})$ onto $\partial \H^2$ defined by $T \mapsto p$ where $p$ is the unique point which is contained in the geodesic extension of every ``vertical'' edge of tiles in $T$. If there is a $\PSL_2(\R)$-invariant Borel probability measure on $\sT(\{\sigma\})$ then it pushes forward to a $\PSL_2(\R)$-invariant measures on $\partial \H^2$. But, an easy exercise shows that there are no $\PSL_2(\R)$-invariant Borel probability measures on $\partial \H^2$.

To get an invariant measure we will need another tile. Consider the standard horoball packing as shown in figure \ref{fig:horoball} in the Poincar\'e model. It
is invariant under $PSL_2(\Z) < PSL_2(\R)$. Let $\tau$ be one of the
curvilinear triangles in the complement of the horoballs. Let $w$ be the length of one of its edges.

Let $P=\{\tau,\sigma_w\}$. $P$ tiles in the following way. Consider the standard horoball packing. Tile each horoball with copies of $\sigma_w$. Use the triangle $\tau$ to tile the complement of the horoballs. See figure \ref{fig:aperiodic}. In [BHRS] it was shown that $P$ is aperiodic and there are uncountably many ergodic $\PSL_2(\R)$-invariant Borel probability measures on $\sT(P)$. Here we only sketch aperiodicity and existence of a Borel probability measure.

\begin{figure}[htb]
\begin{center}
 \psfig{file=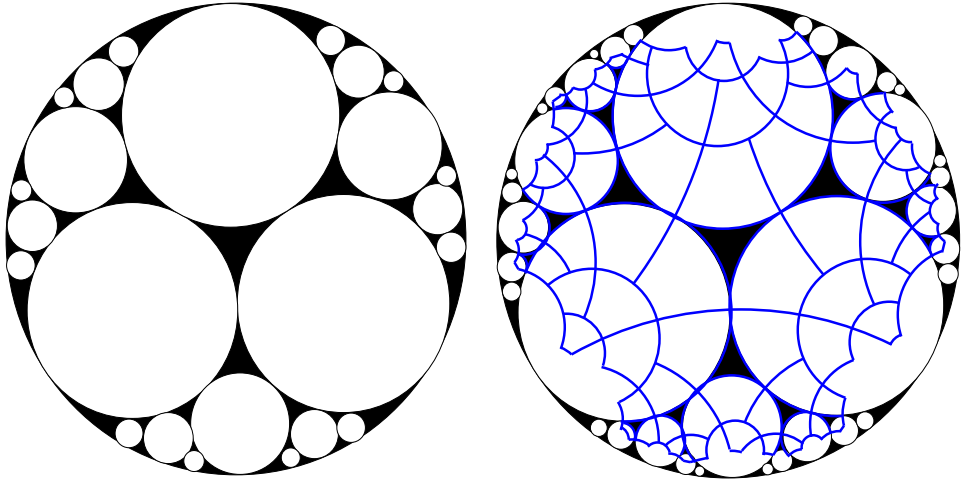,height=2.5in,width=5in}
\caption{Left: Part of the horoball packing invariant under $PSL_2(\Z)$. Right: Part of a tiling by $\{\tau, \sigma_w\}$. }
\label{fig:horoball}\label{fig:aperiodic}
\end{center}
\end{figure}

Let $T \in \sT(P)$. Note that $T$ contains a copy of
$\sigma_w$ since it is not possible to tile with copies of the triangle $\tau$ alone. Let
$\sigma'$ be such a copy. Note that only another copy of
$\sigma_w$ is allowed to be next to $\sigma'$ on either its left or its right edges.
Therefore, $T$ contains a copy of the set $\{t^i \sigma_w | i \in \Z\}$ where $t$ is the
isometry $t(z)=z+w$. So if $e$ is the ``top'' edge of $\sigma'$ then the set $h=\{t^i e | \, i \in \Z\}$ is a horocycle. So there is a horocycle contained in the edges of $T$. 

If $\Gamma$ is the symmetry group of $T$ then the quotient space $\H^2/\Gamma$ admits a
tiling by $\{\sigma,\tau\}$ (the quotient tiling). So the horocycle $h$ must descend to a closed
horocycle on $\H^2/\Gamma$. Thus $\H^2/\Gamma$ is noncompact. But any tile placed deep
enough in a cusp necessarily has self-intersections. So $\H^2/\Gamma$ has no cusps. So
it cannot have finite volume. This implies $T$ is nonperiodic. Since $T$ is arbitrary, $P$ is aperiodic.

To produce a $\PSL_2(\R)$-invariant Borel probability measure on $\sT(P)$, consider the standard horoball packing as in figure \ref{fig:horoball}. Use the triangle $\tau$ to tile the complement of the union of horoballs. For $N>0$, tile the first $N$ ``rows'' of each horoball with copies of $\sigma_w$. This can be done in such a way that the resulting partial tiling has symmetry group $H_N$ with finite index in $PSL_2(\Z)$. Hence $H_N$ is a lattice. So there is an invariant probability measure $\mu_N$ supported on the translates of this partial tiling.

The set of partial tilings of $\H^2$ by $P$ is topologized in a manner analogous to how $\sT(P)$ is topologized. With this topology, it is compact and metrizable. So the Banach-Alaoglu theorem implies the existence of a weak* limit point of the sequence $\{\mu_N\}$. Let $\mu$ be such a point. Given any fixed point $p \in \H^2$, the $\mu_N$-probability that $p$ is contained in a tile tends to $1$ as $N \to \infty$. Hence $\mu$ is supported on full tilings (as opposed to partial tilings). Since each $\mu_N$ is $\PSL_2(\R)$-invariant, $\mu$ is also $\PSL_2(\R)$-invariant.
\end{proof}


Recall that the {\it support} of a measure $\mu$ on a topological space $X$ is complement of the largest open subset $O$ with $\mu(O)=0$. It is denoted here by $\supp(\mu)$. The action of a group $G$ on $X$ is {\it minimal} if every orbit $Gx$ is dense in $X$.

In the example provided above, $\PSL_2(\R)$ does not act minimally on the support of any invariant probability measure $\mu$ on $\sT(P)$. Indeed, if $T$ is any tiling by $P$ then there exists a tiling $T'$ in the closure of the $\PSL_2(\R)$-orbit of $T$ that is a tiling by $\{\sigma\}$ alone. In fact, $\sT(\{\sigma\})$ is contained in the closure of the $\PSL_2(\R)$-orbit of $T$. $\PSL_2(\R)$ acts minimally on $\sT(\{\sigma\}) \subset \sT(P)$, but no $\PSL_2(\R)$-invariant Borel probability measure has support in $\sT(\{\sigma\})$. I conjecture that this phenomenon holds for every finite aperiodic tile set:
\begin{conj}\label{conjecture}
Let $Q$ be an aperiodic finite set of tiles of $\H^2$. Suppose there exists a $\PSL_2(\R)$-invariant Borel probability measure $\mu$ on $\sT(Q)$. Then $\PSL_2(\R)$ does not act minimally on the support of $\mu$.
\end{conj}
Recall that a {\it surface group} is the fundamental group of a closed surface of genus at least 2. Here is a discrete form of the same conjecture:



\begin{conj}\label{conjecture2}
Let $X \subset V^{\Sigma}$ be the graph subshift defined by a {\em finite} graph $\sG$ where $\Sigma$ is a surface group. If there exists a shift-invariant Borel probability measure $\mu$ on $X$ such that $\Sigma$ acts minimally on the support of $\mu$ then there exists a periodic point $x \in X$ (i.e., the stabilizer of $x$ has finite index in $\Sigma$). Moreover, if for some $v\in V$, $\mu\big(\{x \in X|x(id)=v\}\big)>0$ then there exists a periodic point $x\in X$ with $x(id)=v$.
\end{conj}

\begin{thm}
If the conjecture above is true, then the surface subgroup conjecture is true. I.e., if $\Gamma<\PSL_2(\C)$ is a cocompact discrete group then there exists a subgroup $\Sigma <\Gamma$ that is isomorphic to the fundamental group of a closed surface of genus at least 2.
\end{thm}

\begin{proof}
Let $\Sigma<\PSL_2(\R)$ be a discrete cocompact surface group. Let $Y$ be any subset of $\Gamma \backslash PSL_2(\C)$ that is $\Sigma$-invariant, closed and such that the action of $\Sigma$ on $Y$ is minimal. We claim that there is a $\Sigma$-invariant probability measure supported on $Y$.

 Consider the set 
$$\tY = \Big\{(y g, \Sigma g) \in \Gamma \backslash PSL_2(\C) \times \Sigma  \backslash PSL_2(\R)~|~ y\in Y, ~g\in PSL_2(\R)\Big\}.$$
We claim that this set is closed. To see this, let $\{(y_ig_i, \Sigma g_i)\}_{i=1}^\infty$ be a sequence in $\tY$ where $\{y_i\} \subset Y$ and $\{g_i\} \subset PSL_2(\R)$. Let $D \subset PSL_2(\R)$ be a compact set such that $\Sigma D=PSL_2(\R)$. Then there exists $h_i \in \Sigma $ and $d_i \in D$ such that $g_i=h_id_i$. After passing to a subsequence we may assume that $\{y_ih_i\} \subset Y$ has a limit point $y_\infty$ and that $\{d_i\}$ has a limit point $d_\infty \in D$. Thus $\{(y_ig_i,\Sigma g_i)\}=\{(y_ih_id_i,\Sigma d_i)\}$ has the limit point $(y_\infty d_\infty, \Sigma d_\infty) \in \tY$. This proves that $\tY$ is sequentially compact which implies the claim.

Since $\tY$ is invariant under the diagonal action of $PSL_2(\R)$ on the right,  Ratner's theorems on unipotent flows imply that there there exists a $PSL_2(\R)$-invariant Borel probability measure $\tmu$ supported on $\tY$. Now let $K \subset PSL_2(\R)$  be a fundamental domain for the action of $\Sigma$. That is, $K$ is a Borel set such that $\Sigma K=PSL_2(\R)$ and if $g_1 \ne g_2 \in K$ then $g_1K \cap g_2K$ has Haar measure zero. If $Y_0 \subset Y$ is Borel then define
$$\mu(Y_0) = \tmu\Big( \{ (yk,\Sigma k) \in \tY~|~ y \in Y_0, ~k\in K\}\Big).$$
Then $\mu$ is a $\Sigma$-invariant probability measure supported on $Y$. This proves the claim.

 Let $d$ denote the usual distance function on $\PSL_2(\C)$. Let $S \subset \Sigma$ be a finite symmetric generating set. Let $\rho = \min_{\gamma \in \Gamma-\{id\}, g\in G} d(id,g\gamma g^{-1})$.  Without loss of generality, let us assume that $\Sigma$ has a presentation of the form $\Sigma=\langle S | R \rangle$ where $R$ is a finite set of words in $S$. 

Choose $\epsilon_0$ so that $0<\epsilon_0 \le \epsilon$ and if $r=s_1\cdots s_k$ is in $R$ (with $s_i \in S$) if $g_0,g_1,\ldots,g_{k} \in \PSL_2(\R)$ are such that $d(\phi(s_i),g_i) \le \epsilon_0$ for all $i$ then $d\big( g_1\cdots g_k, id \big) < \rho.$


Let $\delta>0$ be such that for all $g_1, g_2 \in \PSL_2(\C)$ with $d(g_1, id)<\delta_1$ and $d(g_2, id)<\delta_1$ if $s\in S$ then, $d\big(g_1\phi(s)g_2, \phi(s)\big) < \epsilon_0.$


Let $V=\{v_1,v_2,\ldots, v_n\}$ be a Borel partition
of $Y$ into sets $v_i$ of diameter less than
$\delta$. Assume that each $v_i$ has positive $\mu$-measure.


Let $\sG$ be the graph with vertex set $V=\{v_1,\ldots, v_n\}$ and edges defined as follows. For each $v, w \in V$, if there exists elements $p \in v$, $q\in w$ and $s\in S$ such that $p s = q$ then there is a directed edge in $\sG$ from $v$ to $w$ labeled $s$. There are no other edges.

Let $X \subset V^{\Sigma}$ be the graph subshift determined by $\sG$. 


We will choose, for each $x\in X$, an $\epsilon$-perturbation $\phi_x$ of $\phi$. To get started, choose a basepoint $p_i \in v_i$ for each $i$. Assume $p_1 = \Gamma$.

 If there is an edge $e=(v,w)$ in $\sG$ labeled $s\in S$ then there exists points $p \in v, q\in w$ such that $p s=q$. Let $p_v, q_w$ be the basepoints of $v$ and $w$ respectively. Because $v$ and $w$ each have diameter at most $\delta$, there exists elements $g_v, g_w \in \PSL_2(\C)$ such that $d(g_v, id)<\delta$, $d(g_w,id)<\delta$, $p_vg_v=p$ and $q_w g_w=q$.

Let $\psi_e=g_v s g_w^{-1}$. Note that $p_v \psi_e = p_v g_v s g_w^{-1}=q_w$. By choice of $\delta$, $d(\psi_e, s) < \epsilon$.

There is an edge $e'= (w,v)$ in $\sG$ labeled $s^{-1}$. Choose $\psi_{e'}$ so that $\psi_{e'} = \psi_e^{-1}$. 

Let $x \in X$. For $f \in \Sigma$, represent $f$ as $f=t_1\cdots t_m$ for some $t_i \in S$. Let $t_0=id$.

Let $\phi_x(f)=\psi_{e_1} \cdots \psi_{e_m}$ where $e_i$ is the edge from $x(t_0 \cdots t_{i-1})$ to $x(t_1 \cdots t_{i})$ labeled $t_i$. 

To show that $\phi_x$ is well-defined, it suffices to show that if $r=s_1\cdots s_k \in R$ is a relator (with $s_i \in S$) and $e_1,e_2,\ldots,e_k$ is a directed cycle in $\sG$ such that $e_i$ is labeled $s_i$ for all $i$, then $\psi_{e_1} \cdots \psi_{e_k} = id$. Let $v$ be the source of $e_1$ and $p\in v$ its basepoint. Since $e_1,\ldots,e_k$ is a directed cycle, $ p \psi_{e_1} \cdots \psi_{e_k} = p$. Thus, if $p=\Gamma g$ for some $g\in \PSL_2(\C)$, then $\psi_{e_1}\cdots \psi_{e_k} \in g^{-1}\Gamma g$. 

As noted above, $d(\psi(e_i),\phi(s_i)) < \epsilon_0$. So by the choice of $\epsilon_0$, $d\big( \psi_{e_1} \cdots \psi_{e_k} , id \big) < \rho$. By definition of $\rho$, this implies $\psi_{e_1}\cdots \psi_{e_k} = id$. Thus, $\phi_x$ is well-defined.

The rest of the proof is exactly the same as the proof of theorem \ref{thm:uniform} beginning with lemma \ref{lem:start} except in one detail. We must invoke the conjecture above to ensure the existence of a periodic point.

\end{proof}


\begin{thebibliography}{99}


\bibitem[Bo03]{Bo03} L. Bowen, \textit{Periodicity and circle packings of the
hyperbolic plane}.  Geom. Dedicata  102  (2003), 213--236.

\bibitem[BHRS05]{BHRS} L. Bowen, C. Holton, C. Radin, L. Sadun, \textit{Uniqueness and
symmetry in problems of optimally dense packings}.  Math. Phys. Electron. J.  11  (2005),
Paper 1, 34 pp. (electronic).

\bibitem[BH99]{BH99} M. R. Bridson and A. Haefliger. \textit{Metric spaces of non-positive curvature}.
Grundlehren der Mathematischen Wissenschaften [Fundamental Principles of Mathematical Sciences], 319. Springer-Verlag, Berlin, 1999. xxii+643 pp.

\bibitem[Co93]{Co93} Coornaert, M., \textit{Mesures de Patterson-Sullivan sur le Bord d'un Espace hyperbolique au sens de Gromov}, Pacific J. of Math. {\bf 159}, No.2, pp. 241-270, (1993).

\bibitem[Ef53]{Ef53} V. A. Efremovi\v c. \textit{The proximity geometry of Riemannian manifolds}. Uspehi Matem. Nauk (N.S.)  8,  (1953), 189--195.


\bibitem[Gr87]{Gr87} M. Gromov. \textit{Hyperbolic groups}.  Essays in group theory,  75--263, Math. Sci. Res. Inst. Publ., 8, Springer, New York, 1987.

\bibitem[La08]{La08} M. Lackenby. \textit{ Surface subgroups of Kleinian groups with torsion}. arXiv:0804.1309.

\bibitem[LLR08]{LLR08} M. Lackenby, D. Long and A. Reid. \textit{LERF and the Lubotzky-Sarnak conjecture}. Geom. Topol.  12  (2008), 2047--2056.


\bibitem[Ma07]{Masters} J. Masters, \textit{Kleinian groups with ubiquitous surface subgroups}, preprint.


\bibitem[Mi68]{Mi68} J. Milnor. \textit{A note on curvature and fundamental group}.  J. Differential Geometry  2  1968 1--7



\bibitem[Pe78]{Penrose} R. Penrose, \textit{Pentaplexity - a class of non-periodic tilings of the
plane}, Eureka 39(1978) 16-32. (Reproduced in Math. Intell. 2(1979/80) 32-37.)



\bibitem[Sv55]{Sv55} A. S. \v Svarc. \textit{A volume invariant of coverings}. Dokl. Akad. Nauk SSSR (N.S.)  105  (1955), 32--34.

\bibitem[Ve89]{Veech} W.A. Veech, \textit{Teichmüller curves in moduli space, Eisenstein series and an application to triangular billiards}.  Invent. Math.  97  (1989),  no. 3, 553--583.



\end{thebibliography}
\end{document}